\documentclass[10pt]{article}
\usepackage{tikz}
\usepackage[english]{babel}
\usepackage{hyperref}
\usepackage{amsmath,amsfonts,amssymb,amsthm,enumerate,mathabx}
\usepackage{latexsym}
\usepackage{makeidx}
\usepackage{amscd}
\usepackage{hyperref}
\usepackage{vmargin}
\usepackage{cite}
\input xy
\xyoption{all}

\setmarginsrb{23mm}{16mm}{23mm}{16mm}%
            {14mm}{8mm}{14mm}{14mm}

\numberwithin{equation}{section}

\theoremstyle{plain}
\newtheorem{definition}{Definition}[section]
\newtheorem{lemma}[definition]{Lemma}
\newtheorem{theorem}[definition]{Theorem}
\newtheorem{corollary}[definition]{Corollary}

\newtheorem{proposition}[definition]{Proposition}

\theoremstyle{definition}
\newtheorem{example}[definition]{Example}
\newtheorem{remark}[definition]{Remark}

\def\A{\mathcal A}
\def\B{\mathcal B}
\def\C{\mathcal C}

\def\F{\mathcal F}

\def\N{\mathbb N}

\def\U{\mathcal U}

\def\Z{\mathbb Z}

\renewcommand{\hom}{\mathrm{Hom}}
\def\Hom{\mathrm{Hom}}
\def\End{\mathrm{End}}

\def\coker{\mathrm{CoKer}}
\renewcommand\ker{\mathrm{Ker}}
\def\Im{\mathrm{Im}}
\def\Ext{\mathrm{Ext}}
\def\id{id}
\def\Ob{\mathrm{Ob}}
\def\ev{\mathrm{ev}}
\def\ext{\mathrm{ext}}
\def\mor{\mathrm{Mor}}
\def\Flat{\mathrm{Flat}}

\def\Ch{{\bf Ch}}

\def\Add{\mathrm{Add}}
\def\Ab{\mathrm{Ab}}
\def\Ring{\mathrm {Ring}}
\def\CRing{\mathrm {Comm.Ring}}
\def\lFlat{\mathrm{L.Flat}}
\global\def\mod#1{\mathrm{Mod}(#1)}
\global\def\fp#1{\mathrm{mod}\text{-}#1}

\global\def\mod#1{\mathrm{Mod}\text{-}#1}
\global\def\lmod#1{#1\text{-}\mathrm{Mod}}

\global\def\cmod#1{\mathrm{Mod}_{\mathrm{cart}}(#1)}

\title{Cartesian modules over representations of small categories}

\author{Sergio Estrada\footnote{The first named  author was supported by the research grant 18394/JLI/13 of the Fundaci\'on S\'eneca-Agencia de Ciencia y Tecnolog\'{\i}a de la Regi\'on de Murcia in the framework of III PCTRM 2011-2014. Furthermore, he is grateful to the Department of Mathematics of the M.I.T. for its hospitality.}  \and Simone Virili\footnote{The second named author was partially supported by Fondazione Cassa di Risparmio di Padova e Rovigo (Progetto di Eccellenza ``Algebraic structures and their applications'').}}

\begin{document}

\maketitle

\abstract{We introduce the new concept of cartesian module over a pseudofunctor $R$ from a small category to the category of small preadditive categories. Already the case when $R$ is a (strict) functor taking values in the category of commutative rings is sufficient to cover the classical construction of quasi-coherent sheaves of modules over a scheme. On the other hand, our general setting allows for a good theory of contravariant additive locally flat functors, providing a geometrically meaningful extension of Crawley-Boevey's Representation Theorem. As an application, we relate and extend some previous constructions of the pure derived category of a scheme.}



\addtocontents{toc}{\setcounter{tocdepth}{2}}
\tableofcontents

\section{Introduction}

Let $X$ be a small site (that is, a small category whose Grothendieck topology is defined by a pretopology, see \cite{Verdier}). The usual way of defining a ringed site is by considering pairs $(X, \mathcal O_X)$, where $\mathcal O_X$ is a sheaf of commutative rings. More generally a ringed category $(X,\mathcal O_X)$ is a pair such that $X$ is a small category and $\mathcal O_X$ is a presheaf of commutative rings on $X$. Then, the category of presheaves of $\mathcal O_X$-modules on $X$ can be defined. On the other hand, a (not necessarily commutative) ring may be regarded as a special case of a small preadditive category, that is, a small category $R$ such that $R(a,b)$ is an Abelian group, for each $a,b\in \Ob R$, and morphism composition distributes over addition. So the category of modules over a preadditive category naturally arises. There are many sources in the literature which deal with this generalization. Quoting from \cite{rings_with_several_objects}, ``{\em[...] there have been several papers concerned with replacing theorems about rings by theorems about additive\footnote{In \cite{rings_with_several_objects} Mitchell uses the term ``additive" for what is usually known as ``preadditive".} categories. What does not seem to be generally realized is the degree of completeness to which the program can be carried out,[...]\,}". In fact, most facts about module theory have their counterpart in the theory of ``modules over rings with several objects".

This paper attempts to take Mitchell's quote one step further. Indeed, given a small category $\mathcal C$, we take a {\em representation} $R\colon \mathcal C\to \Add$ to be a pseudofunctor from $\C$ to the category of small preadditive categories (see Definition \ref{def.repres}). In case $R\colon \C\to \Add$ is indeed a functor, we call $R$ a \emph{strict} representation. In fact, we will be mainly concerned with strict representations in case  $\Add$ is just the category $\Ring$ of (non necessarily commutative) rings. The prototypical example of a strict representation is constructed letting $\C$ be the small category attached to the poset of open affines on a scheme $X$, while the functor $R$ arises from the structure sheaf $\mathcal O_X$ (see also Example \ref{geo_ex}).

Given a representation $R\colon \C\to \Add$, we define the category $\mod R$ of ``right modules" over $R$ (see Definition \ref{def_mod}), whose main properties are studied in Subsection \ref{sub.sec.cat.mod}.

\medskip\par\noindent
{\bf Theorem. }{\em Let $\C$ be a small category and let $R$ be a representation of $\C$. Then $\mod R$ is an $(Ab.4^*)$ Grothendieck category. If furthermore $\C$ is a poset, then $\mod R$ has a projective generator. }
\begin{proof}
See Theorem \ref{theor.mod.Grot}.
\end{proof}

In Subsection \ref{sub.sec.flat}, we define a tensor product in $\mod R$ (taking values in the category of Abelian groups) and hence we get a natural notion of \emph{flatness} in $\mod R$. Its definition yields the following:

\medskip\par\noindent \label{prop_introduction}
{\bf Proposition.} {\em Let $R$ be a representation of the small category $\C$ and let $M$ be a right $R$-module. If $M_c$ is flat in $(R_c^{op},\Ab)$ for all $c\in \Ob \C$, then $M$ is flat. 

If furthermore $\C$ is a poset and $R$ is left flat, then also the converse is true, that is, $M$ is flat in $\mod R$ if and only if $M_c$ is flat in $(R_c^{op},\Ab)$, for all $c\in \Ob\C$.}
\begin{proof}
See Proposition \ref{flat_obj_loc}.
\end{proof}

%%%%%%%%%%%%%%%%%%%%%%%%%%%%%%%%%%%%%%%%%%%5

However (extending work from \cite{SergioQcoh}), in this paper we are mainly concerned with the full subcategory $\cmod R$ of $\mod R$ of the so-called {\em cartesian right $R$-modules} (see Definition \ref{def_cart_mod}). For some special small sites $\mathcal C$, the category of cartesian modules may be regarded as the category ${\rm Qcoh}(R)$ of quasi-coherent modules, where $R$ corresponds to a sheaf of commutative rings on $\mathcal C$. This is certainly the case of the small category $\mathcal C$ of the poset of all open affines on a scheme $X$, when  $R$ is the structure sheaf $\mathcal O_X$.
In this general setting we can prove the following Theorem, recovering \cite[Corollary 4.4]{SergioQcoh} and \cite[Corollary 3.5]{Sergio_Enochs_Rel_homo}:

\medskip\par\noindent
{\bf Theorem.}
{\em Let $\C$ be a small category and consider a right flat representation $R\colon \C\to \Add$ (see Definition \ref{def.flat.rep}). Then, the category $\cmod R$ is Grothendieck.}
\begin{proof}
See Theorem \ref{cart_groth}.
\end{proof}

The abstraction of $\cmod R$ has some nice consequences that we consider in this paper.  For instance it enables to define a ``good" theory of contravariant additive flat functors in a rather general context that encompasses categories of quasi-coherent sheaves of modules. The classical theory of contravariant additive flat functors is related with locally finitely presented additive categories. We recall that a locally finitely presented additive category $\mathcal A$ (in the sense of \cite{Craw})
is an
additive category with direct limits such that every object is a direct
limit of
finitely presented objects, and the class of finitely presented objects is
skeletally small. A sequence $0\to L\stackrel{f}{\to}M
\stackrel{g}{\to}N\to 0$ in $\mathcal A$ is {\it pure} if $$0\to \Hom
(T,L)\to \Hom(T,M)\to \Hom(T,N)\to 0$$ is exact, for each finitely
presented $T$
of $\mathcal A$ (in this case $g$ is said to be a \emph{pure epimorphism} and $f$ a \emph{pure monomorphism}).
Therefore, locally finitely presented additive categories come equipped
with a
canonical notion of flat object (in the sense of Stenstr\"om \cite{S}) and of pure-injective object. Namely,
$F $ is \emph{flat} if every epimorphism $M\to F$ is a pure epimorphism. And $E$ is \emph{pure-injective} if every pure monomorphism $E\to M$ splits.

%%%%%%%%%%%%%%%%%%%%%%%%%%%%%%%%%%%%%%%%%%%%%%%%%%%%%%%%%%%%%%

Let us elaborate a bit more on the classical representation theory of locally finitely presented additive categories. We start from the case of the category of right $R$-modules ($R$ associative ring with identity), which is certainly locally finitely presented and additive.
The subcategory $\fp R$ of finitely presented right $R$-modules is a skeletally small additive category in which idempotents split. A representation theorem of Crawley-Boevey \cite[Theorem 1.4]{Craw} asserts, in particular, that the category $\mod R$ of right $R$-modules is equivalent to the category $\Flat((\fp R)^{op},\Ab) $ of contravariant additive flat functors from $(\fp R)^{op}$ to the category of Abelian groups. Under such equivalence, the pure short exact sequences correspond to short exact sequences in $\Flat((\fp R)^{op},\Ab)$. Similarly, pure-injective modules correspond to flat functors $F$ such that every short exact sequence in $\Flat((\fp R)^{op},\Ab)$ whose first non-zero term is $F$ splits (see \cite{Herzog}). Modules satisfying this homological condition are called cotorsion. They were first introduced by Harrison in \cite{Harr} for Abelian groups as a homological generalization of the algebraically compact Abelian groups.

For a more general locally finitely presented additive category $\A$, one obtains an equivalence between $\A$ and the category $\Flat({\rm fp}(\A)^{op},\Ab)$ of contravariant additive flat functors from the additive and skeletally small category of finitely presented objects of $\A$ into Abelian groups. On the other hand, the category $({\rm fp}(\A)^{op},\Ab)$ of all contravariant additive functors, may be regarded as the category $\mod S$, the category of all unitary right $S$-modules, for a certain ring $S$ with enough idempotents; under this identification, $\Flat({\rm fp}(\A)^{op},\Ab)$ corresponds to the class of all unitary flat right $S$-modules $\Flat(S)$. Consequently, the study of locally finitely presented additive categories is encoded in the study of the category of flat right $S$-modules for a ring with enough idempotents (see \cite{GuilHer,GuilHer2,GuilHer3} for an extension to flat cotorsion modules of the classical theory of pure-injective modules). As an application of this techniques, Herzog showed in \cite{Herzog} the existence of pure-injective envelopes for any object in a locally finitely presented additive category.

For many schemes $X$ that occur in practice (e.g., quasi-compact and quasi-separated), it is known that the category of quasi-coherent modules on $X$ is a locally finitely presented Abelian category (see \cite[I.6.9.12]{GD}, and \cite[Proposition 7]{Garkusha} for the explicit statement). This is also the case for the category of quasi-coherent modules on a Noetherian stack (see \cite[Lemma 3.9]{Lurie}) and on a concentrated Deligne-Mumford stack (see \cite[Theorem A and Proposition 2.7]{Rydh}). However the general theory explained above, when applied to this setting, confronts a major problem: the usual (categorical) notion of purity in a locally finitely presented additive category is not in general well-behaved in categories of sheaves. As an evidence of this, it was shown in \cite[Theorem 4.4]{ES} that for a wide class of projective schemes the only flat object coming from the categorical notion of purity is the zero sheaf. This suggests that a different notion of purity, better reflecting the local nature of the scheme, should be considered. The papers \cite{EEO,EGO} deal with purity in terms of modules of sections rather than categorical purity. This ``local" notion of purity gives back the classical notion of flat sheaf in terms of the stalks (or equivalently  in terms of  flatness of the module of sections on each open affine); one can show that, unless the scheme is affine, this notion of flatness does not coincide with the categorical flatness described above. Working with the local notion of purity, it was shown in \cite{EEO} that pure injective envelopes exist, and in \cite{EGO} that a good notion of pure derived category on a quasi-separated scheme can be defined from an injective Quillen model category structure in the category of unbounded complexes of quasi-coherent modules. From that perspective, this paper continues the ongoing program on the study of the  properties of this purity by presenting a new version of the classical representation theorem of locally finitely presented additive categories in terms of contravariant additive flat functors.

More precisely, we exhibit a ``local" version of the Yoneda extension functor that allows to embed $\cmod R$ (with $R\colon \C\to \Ring$ a strict representation) into the category of cartesian modules over a new representation $R_{fp}\colon \C\to \Add$ defined as follows (see Section \ref{sect.repres} and, in particular, Definition \ref{ind_rep_def}):

\medskip\par\noindent
{\bf Definition.} 
{\em Let $R\colon \C\to \Ring$ be a strict representation. We define the pseudofunctor $R_{fp}\colon \C\to \Add$ as follows:
\begin{enumerate}[\rm --]
\item $R_{fp}(c)=\fp {R_c}$, for all $c\in \Ob\C$;
\item $R_{fp}(\alpha)=-\otimes_{R_c}R_d\colon \fp {R_c}\to \fp {R_d}$, for all $(\alpha\colon c\to d)\in \C$;
\item given $(\alpha\colon c\overset{}{\to} d)$ and $(\beta\colon d\overset{}{\to} e)\in\C$, we let $\mu_{\beta,\alpha}\colon R_{fp}(\beta)R_{fp}(\alpha)\to R_{fp}(\beta\alpha)$ be the natural isomorphism such that, for all $F\in \fp {R_c}$, 
\begin{align*}\mu_{\beta,\alpha}\colon R_{fp}(\beta)R_{fp}(\alpha)F=(F\otimes_{R_{c}}R_d)\otimes_{R_d}R_e&\longrightarrow F\otimes_{R_c}R_e=R_{fp}(\beta\alpha)F\\
(f\otimes r_1)\otimes r_2&\longmapsto f\otimes R_\beta(r_1)r_2\,;\end{align*}
\item given $c\in \Ob\C$, $\delta_c\colon \id_{R_{fp}(c)}\to R_{fp}(\id_c)$ is the natural isomorphism such that, for all $F\in\fp {R_{c}}$,
\begin{align*}
(\delta_c)_F\colon \id_{R_{fp}(c)}F=F&\longrightarrow F\otimes_{R_c}R_c=R_{fp}(\id_c)\\
f&\longmapsto f\otimes 1_{R_c}\,.
\end{align*}
\end{enumerate}}

\bigskip

In this way, we obtain a new category of cartesian modules over $R_{fp}$. It would be desirable for this category to share the same categorical properties as the original $\cmod R$. This turns out to be case:

\medskip\par\noindent
{\bf Lemma.}
{\em In the notation of the previous definition, $R$ is a right flat representation if and only if $R_{fp}$ is right flat. Consequently, for such $R$, both $\cmod R$ and $\cmod {R_{fp}}$ are Grothendieck categories.}

\begin{proof}
See Lemma \ref{always_left_flat}.
\end{proof}
%To see that $\cmod {R_{fp}}$ is truly a well defined Grothendieck category takes some effort. Section 2 of the paper is devoted to pursue this goal. So we define the %category of (cartesian) modules $\cmod R$ over a representation $R:\C\to \Add$ and show in Theorem 2.19 that it is a Grothendieck category, provided that the %representation $R:\C\to \Add$ is right flat (see Definition 2.3). In particular we recover \cite[Corollary 4.4]{SergioQcoh} and \cite[Corollary 3.5]{Sergio_Enochs_Rel_homo}.

Let $\lFlat(\cmod {R_{fp}})$ be the full subcategory of $\cmod {R_{fp}}$ whose objects are the cartesian modules $M$ such that $M_c\in \Flat((\fp {R_c})^{op},\Ab)$, that is, $M_c$ is a contravariant additive flat functor in $((\fp {R_c})^{op},\Ab)$, for each $c\in \Ob\C$ (see Proposition on page \pageref{prop_introduction}). 
Our main theorem asserts that the category of cartesian modules over a strict representation $R\colon \C\to \Ring$ is equivalent to $\lFlat(\cmod {R_{fp}})$. Notice that the definition of flatness used here is the analog in $\cmod {R_{fp}}$ to the definition of a flat quasi-coherent sheaf. Notice also that the ambient category  $\cmod {R_{fp}}$ does not have in general enough projectives (as opposed to the ambient category $({\rm fp}(\A)^{op},\Ab)$ of additive contravariant functors associated to a locally finitely presented additive category $\mathcal A$). Finally, the categories $\lFlat(\mod {R_{fp}})\subseteq \cmod {R_{fp}}$ need not be locally finitely presented (again, as opposed to the usual additive contravariant functor categories $\Flat({\rm fp}(\A)^{op},\Ab)\subseteq ({\rm fp}(\A)^{op},\Ab)$).
\bigskip\par\noindent
{\bf Representation Theorem.}
{\em Let $\C$ be a small category and let $R\colon \C\to \Ring$ be a strict representation. Then the Yoneda functor $Y\colon \mod R\to \mod R_{fp}$ induces equivalences
$$\mod R\cong \lFlat(\mod {R_{fp}})\ \  \text{ and }\ \ \cmod R\cong \lFlat(\cmod {R_{fp}})\,.$$}
\begin{proof}
See Theorem \ref{locally_CB_rep}.
\end{proof}

%It is interesting to point out that, in case $\mathcal C$ is a poset and $R$ is a left flat representation, the class $\lFlat(\cmod R)$ coincides with the class of %cartesian $R$-modules that are flat as a modules in $\mod R$ (see Proposition \ref{flat_obj_loc}). This extends the well-known case of flat quasi-coherent sheaves on a %scheme.

Under this equivalence a short exact sequence $0\to L\to M\to N\to 0$ in $\cmod R$ is pure (that is, $0\to L_c\to M_c\to N_c\to 0$ is pure in $\mod R_c$, for each $c\in \C$) if and only if $0\to Y(L)\to Y(M)\to Y(N)\to 0$ is exact in  $\lFlat(\cmod {R_{fp}})$. Consequently, the study of pure-injective objects in $\cmod R$ is equivalent to the study of flat cotorsion objects in $\lFlat(\cmod {R_{fp}})$. One main advantage of this is the following: it easily follows by the results of \cite{Sergio_Enochs_Rel_homo} that the category $\cmod {R_{fp}}$  admits covers with respect to the class $\lFlat(\cmod {R_{fp}})$. A standard argument then yields the existence of cotorsion envelopes in $\cmod {R_{fp}}$. In fact if $Y(M)\in \lFlat(\cmod {R_{fp}})$ then its cotorsion envelope lies in $\lFlat(\cmod {R_{fp}})$ and it corresponds to the pure-injective envelope of $M$ in $\cmod R$. Thus we get the following:

\bigskip\par\noindent
{\bf Corollary.}
{\em Let $\C$ be a small category and let $R\colon \C\to \Ring$ be a strict representation. Then, every $M\in \cmod R$ has a pure-injective envelope.}

\bigskip\par\noindent
\begin{proof} 
See Theorem \ref{theor.purei.env}.
\end{proof}
Another application of the Representation Theorem is that the pure homological algebra we just defined in $\cmod R$ is encompassed by Quillen's homotopical algebra. That is, the category of unbounded chain complexes of objects in $\cmod R$ has a model structure that reproduces the pure homological algebra. In \cite[\S 1. Corollary]{EGO} the authors define the pure derived category on a quasi-separated scheme. There, they also give an alternative approach for defining the pure derived category in \cite[\S 1. Theorem C]{EGO} by showing that there is an injective model category structure on the category of unbounded chain complexes of flat sheaves. This second approach is the one followed by Murfet and Salarian in \cite{MS} to define the pure derived category on a Noetherian separated scheme. We point out that these pure derived categories are necessarily different from the ones introduced in \cite{Krause} and \cite{CH} (as they make reference to a different notion of purity).
Our Representation Theorem enables us to infer that the two constructions are related in the sense that the pure derived category of a scheme $X$ in \cite[\S 1. Corollary]{EGO} is the homotopy category of the injective model category structure on the category of unbounded chain complexes of locally flats in $\cmod {R_{fp}}$ (i.e. the modules in $\lFlat (\cmod {R_{fp}})$), where $R_{fp}$ is the representation associated  to the structure sheaf of $X$. Hence, we can define the pure derived category in the general setting of cartesian modules on a strict representation of rings so, in particular, with no restriction on the scheme.

\bigskip\par\noindent
{\bf Corollary}.
{\em Let $\C$ be a small category and let $R\colon \C\to \Ring$ be a strict representation. There exists an exact Quillen model category structure in the category of unbounded complexes of objects in $\cmod R$ where the weak equivalences are the pure quasi-isomorphisms. Its corresponding homotopy category is the pure derived category of cartesian $R$-modules.}
\bigskip\par\noindent
\begin{proof} 
See Theorem \ref{theor.der.pur}.
\end{proof}

% and since $\Flat(\cmod {R_{fp}})$ is closed under cotorsion envelopes,   Then it is possible to use known theorems to infer, for instance, the existence of pure-injective %envelopes
%in $\cmod R$ from the fact that cotorsion 

%The theorem 
%One of the most interesting features of this equivalence is that 
% is such that transforms into short exact sequences in $\cmod {R_{fp}}$ that 
%One consequence of this theorem is then the existence of pure injective 

%Thus Sections 2 and 3 of the paper are devoted to  

\section{Preliminaries: Rings with several objects and their modules}

In this paper we will be concerned with small preadditive categories and modules over them. In what follows we state some basic results that will be needed in later sections.

\subsection{Generalities}
\subsubsection{Preadditive categories}

We denote by $\Add$ the category of {\em small preadditive categories}, while we denote by $\Ring$ (resp., $\CRing$) the full sub-category of $\Add$ of unitary associative (commutative) rings. By $R$ we  usually denote a small preadditive category. By $r\in R$ we  mean that $r$ is a {\em morphism} in $R$, while we  use the notation $a\in \Ob R$ when $a$ is an object in $R$. 

\begin{definition}
Let $R$ and $S$ be two small preadditive categories. The {\em tensor product} $R\otimes S$ is the following small preadditive category:
\begin{enumerate}[\rm --]
\item the set of object $\Ob(R\otimes S)$ is just $\Ob(R)\times \Ob(S)$;
\item given $(r_1,s_1),\, (r_2,s_2)\in\Ob(R\otimes S)$ we let 
$$\hom_{R\otimes S}((r_1,s_1),(r_2,s_2))=\hom_R(r_1,r_2)\otimes_\Z\hom_S(s_1,s_2)\,;$$
\item composition of morphisms in $R\otimes S$ is defined by the following law:
$$(f_1\otimes g_1)\circ (f_2\otimes g_2)=(f_1\circ f_2)\otimes (g_1\circ g_2)\,.$$
\end{enumerate}
\end{definition}
Using the properties of the tensor product of Abelian groups, it is not difficult to show that the tensor product $R\otimes S$ in the above definition is again a preadditive category. One can also show that the tensor product of preadditive categories is associative using the same property for the tensor product of Abelian groups.

\subsubsection{Modules} 
A {\em right} (resp., {\em left}) {\em module} $M$ over a small preadditive category $R$ is a (always additive) functor $M\colon R^{op}\to \Ab$ (resp., $M\colon R\to \Ab$). A morphism (a natural transformation) $\phi\colon M\to N$ between right $R$-modules consists of a family of morphisms
$\phi_a\colon M(a)\to N(a)$ (of Abelian groups), with $a$ ranging in $\Ob R$, such that the following squares commute for all $(r\colon a\to b)\in R$:
$$
\xymatrix{
M(a)\ar[r]^{\phi_a}&N(a)\\
M(b)\ar[u]^{M(r)}\ar[r]^{\phi_b}&N(b)\ar[u]_{N(r)}
}
$$
We denote by  $(R^{op},\Ab)$ (resp., $(R,\Ab)$) the category of right (resp., left) $R$-modules. We often say just module to mean right $R$-module when no confusion is possible. When $R$ is a ring we usually write $\mod R$ and $\lmod R$ for $(R^{op},\Ab)$ and $(R,\Ab)$, respectively. 

\begin{lemma}\label{mod_is_groth}
Let $R$ be a small preadditive category. The category of modules $(R^{op},\Ab)$ is a Grothendieck category with a family of small projective generators. 
\end{lemma}

The above lemma is well-known, but let us give some hint for the proof. Indeed, the zero object $0$ in $(R^{op},\Ab)$ corresponds to the trivial functor $R^{op}\to \Ab$, which sends any object to $0$ and any morphism to the zero morphism. Given a morphism $\phi\colon M\to N$ in $(R^{op},\Ab)$, the (co)kernel of $\phi$ is constructed sending $a\in\Ob R$ to the (co)kernel of $\phi_a$ and sending $(r\colon a\to b)\in R$ to the unique map $\ker(\phi_a)\to \ker(\phi_b)$ (resp., $\coker(\phi_a)\to \coker(\phi_b)$) given by the universal property of (co)kernels. By this description we  see that a sequence $0\to N\to M\to M/N\to 0$ in $(R^{op},\Ab)$ is short exact if and only if $0\to N(a)\to M(a)\to (M/N(a))(=M(a)/N(a))\to 0$ is a short exact sequence in $\Ab$, for all $a\in \Ob R$. \\
Analogously, arbitrary products and coproducts are induced componentwise by the products and coproducts in $\Ab$. By this description we see that $(R^{op},\Ab)$ is a
bicomplete Abelian category with exact products and exact colimits.\\
To see that $(R^{op},\Ab)$ is Grothendieck, it remains to describe a family of generators. In fact more is true, that is, $(R^{op},\Ab)$ has a family of small projective generators that correspond to the corepresentable functors
$$\hom_R(-,a)\colon R^{op}\to \Ab$$
with $a\in\Ob R$. In what follows we  use the notation
$$\hom_R(-,a)=H_a\,.$$
Let $M$ be a right $R$-module. By the Yoneda Lemma, there is an isomorphism of Abelian groups 
$$\hom_{(R^{op},\Ab)}(H_a,M)\cong M(a)\,,$$ for all $a\in \Ob R$. This justifies the following definition:

\begin{definition}
Let $R$ be a small preadditive category and let $M$ be a right $R$-module. Given $a\in \Ob R$, an {\em $a$-element} of $M$ is a morphism $m\colon H_a\to M$, while an {\em element} of $M$ is an $a$-element for some $a\in\Ob R$. We write $m\in M$ to mean that $m$ is an element of $M$. Furthermore, we let 
$$|M|=\sum_{a\in \Ob R}|\hom(H_a,M)|=\sum_{a\in \Ob R}|M(a)|$$
\end{definition}

The following lemma comes from the analogous properties in $\Ab$, so we omit its proof.

\begin{lemma}\label{card_lemma}
Let $R$ be a small preadditive category. Then, 
\begin{enumerate}[\rm (1)]
\item given a short exact sequence $0\to N\to M\to M/N\to 0$, in $(R^{op},\Ab)$,
$$|M|=|N|\cdot |M/N|\leq \max\{|\N|,|N|,|M/N|\}\,.$$
\item given a set $I$ and $M_i\in (R^{op},\Ab)$ for all $i\in I$,
$$\left|\bigoplus_{i\in I}M_i\right|=\sup\left\{\sum_{i\in F}|M_i|:F\subseteq I \text{ finite}\right\}\leq \sup\{|\N|,|I|,|M_i|:i\in I\}\,.$$
\end{enumerate}
\end{lemma}

Let $R$ be a small preadditive category, let $M$ be a right $R$-module, and let $N\leq M$ be a submodule. By the description of the Abelian structure given after Lemma \ref{mod_is_groth}, we can identify $N(a)$ with a subgroup of $M(a)$, for all $a\in \Ob R$. Suppose now that $M'$ is another right $R$-modules and let $N'\leq M'$. Given a morphism $\phi\colon M\to M'$, we say that $\psi\colon N\to N'$ is a restriction of $\phi$ if the following square commutes
$$
\xymatrix{
M\ar[r]^{\phi}&M'\\
N\ar[u]\ar[r]^{\psi}&N'\ar[u]
}
$$
where the vertical arrows are the canonical inclusions of $N\to M$ and $N'\to M'$.

\begin{lemma}\label{restriction_map}
In the above notation, the following are equivalent
\begin{enumerate}[\rm (1)]
\item a restriction $\psi\colon N\to N'$ exists;
\item for all $a\in\Ob R$, $\phi_a(N_a)\subseteq N_a'$;
\item for all $a\in\Ob R$ and all $f\colon H_a\to N\to M$, there exists $g\colon H_a\to N'\to M'$  such that $\phi\circ f=g$.
\end{enumerate}
\end{lemma}
%\begin{proof}
%Indeed, just notice that condition \eqref{rest} is equivalent to say that, for all $a\in\Ob R$, given $x\in M'_a$ there exists a (necessarily unique) $y\in N'_a$ such that $\phi_a((\iota_M)_a(x))=(\iota_N)_a(y)$, equivalently, $\phi_a((\iota_N)_a(M'_a))\subseteq N'_a$, that is, each map $\phi_a:M_a\to N_a$ can be restricted to a morphism $\phi'_a:M'_a\to N_a'$. It is now easy to check that we can define a morphism $\phi':M'\to N'$ letting $\phi'_a:M'_a\to N'_a$,  for all $a\in\Ob R$. 
%\end{proof}

\subsubsection{Bimodules} Let $R$ and $S$ be two small preadditive categories. A {\em (left $R$)-(right $S$)-bimodule $M$} is a functor
$$M\colon S^{op}\otimes R\to \Ab\,.$$
Notice that $\hom_R(-,-)\colon R^{op}\otimes R\to \Ab$ is a (left $R$)-(right $R$)-bimodule. Furthermore, given a functor $\phi\colon R\to S$, we obtain the following (left $R$)-(right $S$)-bimodule
$$\hom_S(-,\phi(-))\colon S^{op}\otimes R\to \Ab\,.$$
Given a (left $R$)-(right $S$)-bimodule $M\colon S^{op}\otimes R\to \Ab$, for all $a\in \Ob R$ we obtain a right $S$-module
$$M(-,a)\colon S^{op}\to \Ab\,.$$
Similarly, we obtain a left $R$-module $M(b,-)$ for all $b\in \Ob S$.

\medskip
A homomorphism of (left $R$)-(right $S$)-bimodules is the same as a homomorphism of left $R^{op}\otimes S$-modules, so that the category of (left $R$)-(right $S$)-bimodules is naturally equivalent to the category $(R^{op}\otimes S,\Ab)$.

\subsection{Tensor product and flat modules}

\subsubsection{The tensor product functor} Let $R$ be a small preadditive category. As for the case when $R$ has one object (so $R$ is a ring), there is a {\em tensor product functor}
$$(R^{op},\Ab)\times(R,\Ab)\longrightarrow \Ab$$
which satisfies many natural properties. The tensor product can be characterized by a universal property, but we prefer the following more explicit definition. 

\begin{definition}
Let $R$ be a small preadditive category and let $M\in (R^{op},\Ab)$, $N\in (R,\Ab)$, then
$$M\otimes_{R}N=\left(\bigoplus_{a\in\Ob R}M(a)\otimes_{\Z}N(a)\right)/T\,,$$
where $T$ is the subgroup generated by the elements of the form\footnote{Contrarily to $N$, $M$ is contravariant, so that $M(r)\colon M(b)\to M(a)$, while $N(r)\colon N(a)\to N(b)$.}
$$M(r)(x)\otimes y-x\otimes N(r)(y)\,,\ \ \ \ r\in\hom_{R}(a,b)\,,\ \ x\in M(b)\,,\ \ y\in N(a)\,.$$
Furthermore, given two morphisms $\phi\colon M\to M'$ in $(R^{op},\Ab)$ and $\psi\colon N\to N'$ in $(R,\Ab)$, we define the following homomorphism of Abelian groups
$$\phi\otimes_R\psi\colon M\otimes_RN\to M'\otimes_RN'$$
as the morphism induced on the quotient by the diagonal morphism
$$\bigoplus_{a\in\Ob R}\phi_a\otimes_{\Z}\psi_a\colon \bigoplus_{a\in\Ob R}M(a)\otimes_{\Z}N(a)\to\bigoplus_{a\in\Ob R}M'(a)\otimes_{\Z}N'(a)\,.$$
\end{definition}

The following natural properties of the tensor product can be easily verified by hand.

\begin{lemma}\label{basic_tensor}
Let $R$ be a small preadditive category. Then,
\begin{enumerate}[\rm (1)]
\item the tensor product $$-\otimes_R-\colon (R^{op},\Ab)\times (R,\Ab)\to \Ab$$ is cocontinuous (it commutes with colimits) in both variables;
\item $H_a\otimes_R N\cong N(a)$, for all $a\in \Ob R$ and $N\in(R,\Ab)$;
\item letting $H^*_a=\hom_R(a,-)$, $M\otimes_RH^*_a\cong M(a)$, for all $a\in \Ob R$ and $M\in(R^{op},\Ab)$.
\end{enumerate}
\end{lemma}

Let us conclude with the following

\begin{definition}
Let $R$ be a small preadditive category and let $N\leq M\in (R^{op},\Ab)$. We say that $N$ is {\em pure} in $M$ if the sequence $0\to N\otimes_R K\to M\otimes_R K$ is exact for all $K\in (R,\Ab)$.
\end{definition}

The above notion of purity can be ``reduced" to purity in categories of modules over a unitary ring; this reduction, and some of its consequences, is described in the Appendix. 

\subsubsection{Tensor product of bimodules} 
We have seen that the tensor product is a functor $(R^{op},\Ab)\times (R,\Ab)\to \Ab$. If we want a tensor product that takes values in categories of (bi)modules other that $\Ab$ we need to start with categories of bimodules, instead of categories of modules.

\begin{definition}
Let $R$, $S$ and $T$ be small preadditive categories, let $M$ be a (left $T$)-(right $S$)-bimodule and let $N$ be a (left $S$)-(right $R$)-bimodule. We define the (left $T$)-(right $R$)-bimodule $$M\otimes_S N\colon R^{op}\otimes T\to \Ab$$ as follows:
\begin{enumerate}[\rm (1)]
\item for all $t\in\Ob T$ and $r\in \Ob R$, $(M\otimes_S N)(r,t)=M(-,t)\otimes_S N(r,-)$;
\item given morphisms $(r\colon a_1\to a_2)\in R^{op}$ and $(t\colon b_1\to b_2)\in T$, we define 
$$(M\otimes_S N)(r\otimes t)\colon (M\otimes_S N)(a_1,b_1)\to (M\otimes_S N)(a_2,b_2)$$
to be the tensor product over $S$ of the two morphisms $M(t)\colon M(-,b_1)\to M(-,b_2)$ (of right $S$-modules) and $N(r)\colon N(a_1,-)\to N(a_2,-)$ (of left $S$-modules).
\end{enumerate}
Given two homomorphism $\phi\colon M\to M'$ (of (left $T$)-(right $S$)-bimodules) and $\psi\colon N\to N'$ (of (left $S$)-(right $R$)-bimodules) we define the following homomorphism of (left $T$)-(right $R$)-bimodules: 
$$\phi\otimes \psi\colon M\otimes_S M'\to N\otimes_SN'\,,$$
such that $(\phi\otimes \psi)_{(r,t)}=\phi_{(-,t)}\otimes \psi_{(r,-)}$.
\end{definition}

One verifies that the above definition gives a functor
$$-\otimes_S-\colon (S^{op}\otimes T,\Ab)\times  (R^{op}\otimes S,\Ab)\to (R^{op}\otimes T,\Ab)\,.$$

\begin{lemma}\label{assoc_tensor}
Let $R$, $S$, $T$ and $U$ be small preadditive categories, and consider bimodules $M\colon S^{op}\otimes R\to \Ab$, $N\colon T^{op}\otimes S\to \Ab$ and $K\colon U^{op}\otimes T\to \Ab$. There is a natural isomorphism of (left $R$)-(right $U$)-bimodules
$$(M\otimes_S N)\otimes_T K\overset{\sim}{\longrightarrow} M\otimes_S(N\otimes_T K)\,,$$
such that, given $r\in R$ and $u\in U$, the component
$$((M\otimes_S N)\otimes_T K)(r,u)\longrightarrow (M\otimes_S(N\otimes_T K))(r,u)\,,$$
is determined by the assignement $(l\otimes h)\otimes k\mapsto l\otimes (h\otimes k)$, for all $l\in M(r,s)$, $h\in N(s,t)$, $k\in K(t,u)$, $s\in S$, $t\in T$. 
\end{lemma}

\subsubsection{Flat modules}
Let us start with the following

\begin{definition}
Let $R$ be a small preadditive category and let $M\in (R^{op},\Ab)$. Then $M$ is {\em flat} if the functor
$$M\otimes_R-\colon (R,\Ab)\longrightarrow \Ab$$
is exact.
\end{definition}

Given a functor $\alpha\colon \A\to \B$ and  $b\in \Ob\B$, the {\em fiber} of $\alpha$ over $b$ is the category $\A_{/b}$ such that
\begin{enumerate}[\rm --]
\item the objects of $\A_{/b}$ are pairs $(a,f)$, where $a$ is an object in $\A$ and $f\colon \alpha(a)\to b$ is a morphism in $\B$;
\item a morphism $T\colon (a,f)\to (a',f')$ in $\A_{/b}$ is a morphism $T\colon a\to a'$ in $\A$ such that the following diagram commutes in $\B$
$$\xymatrix{\alpha(a)\ar[d]_{\alpha(T)}\ar[rr]^{f}&&  b\\
\alpha(a')\ar@/_10pt/[rru]_{f'}}$$
\end{enumerate}
A non-empty category $\A$ is {\em filtered from above} if it satisfies the following conditions 
\begin{enumerate}[\rm --]
\item given two objects $a,a'\in\A$, there is an object $b\in\A$ such that $\hom_\A(a,b)\neq\emptyset\neq\hom_\A(a',b)$;
\item given two parallel arrows $\phi,\phi'\colon a\rightrightarrows a'$, there is an arrow $\psi\colon a'\to b$, such that $\psi\phi=\psi\phi'$.
\end{enumerate}
The following classical result characterizes flat modules.
\begin{lemma}{\rm \cite{flat_func}}
Let $R$ be a small preadditive category and consider the Yoneda functor
$$Y\colon R\to (R^{op},\Ab)\,.$$
A right $R$-module $M$ is flat if and only if the fiber $R_{/M}$ of $Y$ over $M$ is filtered from above. 
\end{lemma}

\subsection{Change of base}\label{change_of_base_sec}

Let $R$ and $S$ be two small preadditive categories and let $\phi\colon R\to S$ be a functor. The {\em restriction of scalars along $\phi$} is the functor
$$\phi^*=(-\circ \phi)\colon (S^{op},\Ab)\to (R^{op},\Ab)\,,$$
which is defined just composing a functor $N\colon S^{op}\to \Ab$ with $\phi$, to obtain a functor $\phi^*(N)=N\circ\phi\colon R^{op}\to \Ab$. On the other hand, one defines the {\em extension of scalars along $\phi$} 
$$\phi_!=(-\otimes_R\hom_S(-,\phi(-)))\colon (R^{op},\Ab)\to (S^{op},\Ab)\,,$$
just as the tensor product over $R$ by the (left $R$)-(right $S$)-bimodule $\hom_S(-,\phi(-))$. One can show that the extension of scalars is left adjoint to the restriction of scalars. In fact, the unit 
$$\eta_\phi\colon \id_{(R^{op},\Ab)}\to \phi^*\phi_!$$ 
is defined as follows:  for all $M\colon R^{op}\to \Ab$ and all $a\in\Ob R$, the component at $a$ of $(\eta_\phi)_M\colon M\to \phi^*\phi_!M$ is the unique homomorphism of Abelian groups 
$$M_a\to((M\otimes_R\hom_S(-,\phi(-)))\circ \phi)_a=(M\otimes_R\hom_S(-,\phi(-)))_{\phi(a)}=\bigoplus_{b\in \Ob R} M_b\otimes_{\Z}\hom_S(\phi(a),\phi(b))/T$$
which sends $m\in M_a$ to (the image in the quotient of) $m\otimes \id_{\phi(a)}\in M_{a}\otimes_\Z\hom_S(\phi(a),\phi(a))$. Similarly, the counit 
$$\varepsilon_\phi\colon  \phi_!\phi^* \to \id_{(S^{op},\Ab)}$$ 
is defined as follows: for all $N\colon S^{op}\to \Ab$ and all $c\in\Ob S$, the component at $c$ of $(\varepsilon_\phi)_N\colon \phi_!\phi^*N\to N$ is the unique homomorphism of Abelian groups 
$$((N\circ\phi)\otimes_R\hom_S(-,\phi(-)))_c=\bigoplus_{b\in\Ob R}N_{\phi(b)}\otimes_\Z\hom_S(c,\phi(b))/T\to N_c$$
which sends $n\otimes f$ to $N_f(n)\in N_c$, for all $n\in N_{\phi(b)}$, $f\colon c\to \phi(b)$, and $b\in \Ob R$. 

\begin{definition}
Let $R$ and $S$ be two small preadditive categories and let $\phi\colon R\to S$ be an additive functor. The adjunction $(\phi_!,\phi^*)$ is said to be the {\em change of base adjunction along $\phi$}.
\end{definition}

By definition, the scalar restriction $\phi^*$ along any functor $\phi$ is exact. On the other hand, the scalar extension $\phi_!$ is only right exact in general. 

\begin{definition}
Let $\phi\colon R\to S$ be a functor between small preadditive categories. We say that $\phi$ is {\em right flat} if the scalar extension $\phi_!$ along $\phi$ is exact. Similarly, $\phi$ is {\em left flat} if $\phi^{op}\colon R^{op}\to S^{op}$ right flat.
\end{definition}

Let now $\phi\colon R\to S$ and $\psi\colon S\to T$ be functors between preadditive categories. Just by definition, $\phi^*\psi^*=(\psi\phi)^*$, so that there is a natural isomorphism $\psi_!\phi_!\cong (\psi\phi)_!$ (as these functors are adjoint to the same functor). In the following lemma we give an explicit description of such an isomorphism:

\begin{lemma}\label{functoriality}
Let $\phi\colon R\to S$ and $\psi\colon S\to T$ be functors between preadditive categories. Given $M\in(R^{op},\Ab)$, there is a natural isomorphism
\begin{equation}\label{natural_2}(M\otimes_R\hom_S(-,\phi(-)))\otimes_S\hom_T(-,\psi(-))\longrightarrow M\otimes_R\hom_T(-,\psi\phi(-))\end{equation}
defined composing the natural isomorphism 
$$(M\otimes_R\hom_S(-,\phi(-)))\otimes_S\hom_T(-,\psi(-))\to M\otimes_R(\hom_S(-,\phi(-))\otimes_S\hom_T(-,\psi(-)))$$
described in Lemma \ref{assoc_tensor}, with the following natural isomorphism:
\begin{align*}\hom_S(-,\phi(r))\otimes_S\hom_T(t,\psi(-))\overset{\sim}{\longrightarrow}&\hom_T(t,\psi\phi(r))\\
(f\colon s\to \phi(r))\otimes (g\colon t\to \psi(s))\longmapsto&(\psi(f)\circ g\colon t\to\psi\phi(r))\,.\end{align*}
which holds, naturally, for all $r\in \Ob R$ and $t\in \Ob T$. More explicitly, the map in \eqref{natural_2} is defined by the assignements
$$(m\otimes (f\colon b\to \phi(r)))\otimes (g\colon t\to \psi (b))\longmapsto m\otimes (\psi(f)\circ g)\,.$$
\end{lemma}

\subsection{Crawley-Boevey's classical results}

Let $R$ be a ring and denote by $\fp R$ the full subcategory of finitely presented modules in $\mod R$. Consider the {\em contravariant Yoneda functor} 
$$Y\colon \mod R\to ((\fp R)^{op},\Ab) \ \ \text{ such that }\ \ M\mapsto \hom_R(-,M)\restriction_{\fp R}\,,$$
for any $M\in \mod R$ and, similarly, $Y(\phi)=\hom_R(-,\phi)\restriction_{\fp R}$, for any homomorphism $\phi$ of right $R$-modules. Consider also the following {\em evaluation at $R$ functor}, which goes in the opposite direction:
$$Y^{-1}\colon ((\fp R)^{op},\Ab)\to \mod R\ \ \text{ such that }\ \ \mathcal M\mapsto \mathcal M(R)\, ,$$ 
for all $\mathcal M\colon (\fp R)^{op}\to \Ab$, where $\mathcal M(R)\colon R^{op}\to \Ab$ sends $r\in R$ to the morphism $\mathcal M(r\cdot -)\colon \mathcal M(R)\to \mathcal M(R)$, 
and $(r\cdot -)\colon R\to R$ is the endomorphism of $R$ (considered as a right $R$-module) such that $s\mapsto rs$. Similarly, given a natural transformation $\Phi\colon \mathcal M\to \mathcal N$, we let $Y^{-1}(\Phi)=\Phi_R$.
%Since the assignment $r\mapsto \mathcal M(r\cdot -)$ is a ring anti-homomorphism $R\to\End_\Z(\mathcal M(R))$ (being $\mathcal M$ contravariant), it gives  $\mathcal M(R)$ the structure of a right $R$-module. 

\begin{lemma}
In the above notation, $Y^{-1}$ is a left adjoint to $Y$. In fact, 
\begin{enumerate}[\rm (1)]
\item the counit of the adjunction $\varepsilon\colon Y^{-1}Y\rightarrow \id_{\mod R}$ is defined by 
$$\varepsilon_M\colon \hom_R(R,M)\to M\ \ \ \text{such that }\ \ (\phi\colon R\to M)\mapsto \phi(1)\,,$$
for all $M\in\mod R$;
\item the unit of the adjunction $\eta\colon \id_{((\fp R)^{op},\Ab)}\rightarrow YY^{-1}$ is defined by 
$$(\eta_{\mathcal M})_K\colon \mathcal M(K)\to \hom_{R}(K,\mathcal M(R))\,,$$
for all $\mathcal M\colon (\fp R)^{op}\to \Ab$ and $K\in \fp R$, such that 
$$(\eta_{\mathcal M})_K(t)(k)=\mathcal M(k)(t)\,,$$
for all $t\in \mathcal M(K)$ and $k\in K$, where $k$ is identified with the homomorphism $R\to K$ sending $1\mapsto k$.
\end{enumerate}
\end{lemma}

The proof of the above lemma is an exercise and it consists in verifying the counit-unit identities. Let us recall also the following classical result:
%The proof of the above lemma is an exercise, but let us at least sketch its proof. Let $M\in\mod R$ and $\mathcal M:(\fp R)^{op}\to \Ab$ and let us construct an isomorphism
%$$\Phi:\hom_{((\fp R)^{op},\Ab)}(\mathcal M,Y(M))\to \hom_{\mod R}(Y^{-1}(\mathcal M),M)\,.$$
%Indeed, given a natural transformation $F:\mathcal M\to Y(M)$, we let $\Phi(F)=F_R$. It is not difficult to show that $\Phi$ is a natural injective group homomorphism. To see that it is surjective, one should construct and inverse
%$$\Psi:\hom_{\mod R}(Y^{-1}(\mathcal M),M)\to \hom_{((\fp R)^{op},\Ab)}(\mathcal M,Y(M))\,,$$
%such that $\Psi$ assigns to a given morphism $f:Y^{-1}(\mathcal M)\to M$ the natural transformation $\Psi(f):\mathcal M\to Y(M)$ such that, given $K\in\fp R$, $\Psi(f)_K:\mathcal M(K)\to \hom_R(K,M)$ is the group homomorphism such that
%$$ \Psi(f)_K(t)(k)=f(\mathcal M(k)(t))\,,$$
%for all $t\in \mathcal M(K)$ and $k\in K$, where we identified $k$ with the homomorphism $R\to K$ mapping $1\mapsto k$.

\begin{lemma}[Yoneda]
The counit $\varepsilon\colon Y^{-1}Y\rightarrow \id_{\mod R}$ of the above adjunction $(Y^{-1},Y)$ is a natural isomorphism of functors.
\end{lemma}

We can now state the following celebrated result from \cite{Craw}:

\begin{theorem}[Crawley-Boevey]\label{CB_rep_Th}
Given a ring $R$,  the essential image of the contravariant Yoneda functor $Y\colon \mod R\to ((\fp R)^{op},\Ab)$ coincides with the full  subcategory of $((\fp R)^{op},\Ab)$, whose objects are the flat functors. Furthermore, the restriction $Y^{-1}\colon \Flat((\fp R)^{op},\Ab)\to \mod R$ is a quasi-inverse for $Y\colon \mod R\to \Flat((\fp R)^{op},\Ab)$.
\end{theorem}

The following consequence of the above theorem allows for a good ``purity theory" in $\mod R$:

\begin{corollary}\label{purity_th}
Let $R$ be a ring. Then a short exact sequence $0\to N\to M\to M/N\to 0$ is pure exact in $\mod R$ if and only if $0\to Y(N)\to Y(M)\to Y(M/N)\to 0$ is exact in $((\fp R)^{op},\Ab)$. 
\end{corollary}

%Denote by $\Flat((\fp R)^{op},\Ab)$ the full subcategory of $((\fp R)^{op},\Ab)$ whose objects are the flat modules. By the above theorem the restriction of the Yoneda functor $Y:\mod R\to \Flat((\fp R)^{op},\Ab)$, so that it has a quasi-inverse
%$$Y^{-1}:\Flat((\fp R)^{op},\Ab)\to \mod R\,.$$
%Recall also that a quasi-inverse is both a right and a left adjoint. We can describe $Y^{-1}$ quite explicitly, in fact, 
%

\section{Modules over representations of small categories}

\subsection{Representations of small categories}

\begin{definition}\label{def.repres}
Let $\C$ be a small category, a {\em representation} of $\C$ is a pseudofunctor $R\colon \C\to \Add$, that is, $R$ consists of the following data:
\begin{enumerate}[\rm --]
\item for each object $c\in\Ob \C$, a preadditive category $R_c$;
\item for all $c,\,d\in \Ob \C$ and any morphism $\alpha\colon c\to d$, an additive functor $R_\alpha\colon R_c\to R_d$;
\item for any object $c\in \Ob\C$, an isomorphism of functors $\delta_{c}\colon \id_{R_c}\overset{\sim}{\longrightarrow}R_{\id_c} $;
\item for any pair of composable morphism $\alpha$ and $\beta$ in $\C$, an isomorphism of functors $\mu_{\beta,\alpha}\colon R_\beta R_\alpha \overset{\sim}{\longrightarrow}R_{\beta\alpha}$.
\end{enumerate}
Furthermore, we suppose that the following axioms hold
\begin{enumerate}[\rm (Rep.1)]
\item given three composable morphisms $c\overset{\alpha}{\to}d\overset{\beta}{\to}e\overset{\gamma}{\to}f$ in $\C$, the following diagram commutes
$$\xymatrix{
R_{\gamma}R_\beta R_{\alpha}\ar[d]_{\mu_{\gamma,\beta}}\ar[rr]^{R_\gamma(\mu_{\beta,\alpha})}&&R_{\gamma}R_{\beta\alpha}\ar[d]^{\mu_{\gamma,\beta\alpha}}\\
R_{\gamma\beta}R_\alpha\ar[rr]^{\mu_{\gamma\beta,\alpha}}&&R_{\gamma\beta\alpha}
}$$
\item given a homomorphism $(\alpha\colon c\to d)\in\mor \C$, the following diagram commutes
$$\xymatrix{
&R_\alpha\ar@{=}[dd]\ar[dr]^{\delta_d}\ar[dl]_{\delta_{c}}\\
R_\alpha R_{\id_c}\ar[dr]_{\mu_{\alpha,\id_c}}&&R_{\id_d}R_{\alpha}\ar[dl]^{\mu_{\id_d,\alpha}}\\
&R_\alpha
}$$
\end{enumerate}
A representation $R\colon \C\to \Add$ is said to be {\em strict} if it is a functor, that is, $R_{\id_c}=\id_{R_c}$, $R_{\beta\alpha}=R_\beta R_\alpha $, and $\delta$ and $\mu$ are identities. 

\smallskip
Given a representation $R\colon \C\to \Add$ and $(\alpha\colon c\to d)\in\mor\C$, we denote by $$\alpha_!\colon (R_c^{op},\Ab)\rightleftarrows(R_d^{op},\Ab)\colon \alpha^*$$ the change of base adjunction induced by $R_\alpha$.
\end{definition}

We will often consider representations of a small category $\C$ on some full subcategory of $\Add$ as, for example, the category of (commutative) rings.

\begin{example}\label{geo_ex}\cite{SergioQcoh,Sergio_Enochs_Rel_homo}
In this example we list three representations of small categories naturally arising in geometric contexts:
\begin{enumerate}[\rm (1)]
\item let $(X,\mathcal O_X)$ be a scheme, choose an affine open cover $\U$ (e.g., the family of all the affine opens) of $X$ an let $\C$ be the category associated with the poset $\U$, ordered by reverse inclusion. Then there is a natural representation $\mathcal O_X\colon \U\to \CRing$ such that $U\mapsto \mathcal O_X(U)$;
\item let $\mathcal X$ be an algebraic stack with structure sheaf of rings $\mathcal O_{\mathcal X}$.
In this case we consider $\C$ to be a small skeleton of the category of affine
schemes smooth
over $\mathcal X$ and we let $R$ act as the sheaf of rings $\mathcal O_{\mathcal X}$.;
\item let $\mathcal X$ be a Deligne-Mumford stack with structure sheaf of rings $\mathcal O_{\mathcal X}$. We take $\C$ a small skeleton of the
category of affine schemes that are \'etale over $\mathcal X$ (such a small skeleton must exist as
\'etale morphisms are of finite type) and we let $R$ act as the sheaf of rings $\mathcal O_{\mathcal X}$.
\end{enumerate}
\end{example}

Let $\C$ be a small category, and let $R\colon \C\to \Add$ be a representation. Let also $c\overset{\alpha}{\longrightarrow} d\overset{\beta}{\longrightarrow} e$ be morphisms in $\C$, call $\gamma=\beta\alpha$ their composition and consider the base change adjunctions $(\alpha_!,\alpha^*)$, $(\beta_!,\beta^*)$ and $(\gamma_!,\gamma^*)$, relative to the additive functors $R_\alpha\colon R_c\to R_d$, $R_\beta\colon R_d\to R_e$, and $R_\gamma\colon R_c\to R_e$, respectively. In what follows we will describe some precise relations among these three adjunctions. 

\medskip
Recall first the notion of {horizontal} pasting of two natural transformations. Indeed, let $F_1,F_2\colon R\to S$ and $G_1,G_2\colon S^{op}\to T$ be functors, and let $\alpha\colon F_1\rightarrow F_2$ and $\beta\colon G_2\rightarrow G_1$ be natural transformations. The {\em horizontal pasting} $$\alpha*\beta\colon G_2F_2\rightarrow G_1F_1$$ is a natural transformation such that $(\alpha*\beta)_a=G_1(\alpha_a)\circ\beta_{F_2(a)}$,  for all $a\in\Ob R$.

\begin{lemma}\label{change_of_two}
Let $\C$ be a small category, let $R\colon \C\to \Add$ be a representation, let $c\overset{\alpha}{\longrightarrow} d\overset{\beta}{\longrightarrow} e$ be morphisms in $\C$ and let $\gamma=\beta\alpha$. There are natural isomorphisms of functors: 
$$\sigma_{\beta,\alpha}\colon \alpha^*\beta^*\longrightarrow \gamma^*$$
such that $(\sigma_{\beta,\alpha})_M=(\mu_{\beta,\alpha}^{-1}*\id_{N})\colon \gamma^*N\to \alpha^*\beta^*N$, for all $N\colon R_e^{op}\to \Ab$, and 
$$\tau_{\beta,\alpha}\colon \beta_!\alpha_!\longrightarrow \gamma_!\,,$$
where, for all $M\colon R_c^{op}\to \Ab$, $(\tau_{\beta,\alpha})_M\colon \beta_!\alpha_!M\to \gamma_!M$ is the following composition:
$$(M\otimes_{R_c}\hom_{R_d}(-,R_\alpha))\otimes_{R_d}\hom_{R_e}(-,R_\beta)\overset{(*)}{\to} M\otimes_{R_c}\hom_{R_e}(-,R_\beta R_\alpha)\overset{(**)}{\to}  M\otimes_{R_c}\hom_{R_e}(-,R_\gamma)\,,$$
where the map $(*)$ is described in Lemma \ref{functoriality}, and $(**)$ is the map $\id_M\otimes_{R_c} \hom_{R_e}(-,(\mu_{\beta,\alpha})_{(-)})$. 
\end{lemma}

Notice that, for a strict representation, $\sigma_{\beta,\alpha}$ is just the identity, while $\tau_{\beta,\alpha}$ reduces to the isomorphism described in Lemma \ref{functoriality}.

\smallskip
Let us now use the natural isomorphisms constructed in the above lemma to relate units and counits relative to different change of base adjunctions. The proof of the following lemma consists in checking the statement on elements.

\begin{lemma}\label{unit_of_two}
Let $\C$ be a small category, let $R\colon \C\to \Add$ be a representation, let $c\overset{\alpha}{\longrightarrow} d\overset{\beta}{\longrightarrow} e$ be morphisms in $\C$ and let $\gamma=\beta\alpha$. Denote by $\eta_\alpha$, $\eta_\beta$, and $\eta_\gamma$ (resp., $\varepsilon_\alpha$, $\varepsilon_\beta$, and $\varepsilon_\gamma$) be the units (resp., counits) of the change of base adjunctions $(\alpha_!,\alpha^*)$, $(\beta_!,\beta^*)$ and $(\gamma_!,\gamma^*)$, respectively. Then there are commutative squares:
$$
\xymatrix{
\beta_!\alpha_!\alpha^*\beta^* \ar[rr]^{\beta_!(\varepsilon_\alpha)_{\beta^*}}\ar[d]_{\beta_!\alpha_!(\sigma_{\beta,\alpha})}&&
\beta_!\beta^*\ar[dd]^{\varepsilon_\beta}&&&
\id_{(R_c^{op},\Ab)}\ar[rr]^{\eta_\gamma}\ar[dd]_{\eta_\alpha}&&
\gamma^*\gamma_!\ar[d]^{\gamma^*(\tau_{\beta,\alpha}^{-1})}
\\
\beta_!\alpha_!\gamma^*\ar[d]_{(\tau_{\beta,\alpha})_{\gamma^*}}&&
 &&&
 &&
 \gamma^*\beta_!\alpha_!\ar[d]^{(\sigma_{\beta,\alpha}^{-1})_{\alpha^*\beta^*}}
 \\
\gamma_!\gamma^*\ar[rr]^{\varepsilon_\gamma}&&
\id_{(R_e^{op},\Ab)}&&&
\alpha^*\alpha_! \ar[rr]^{\alpha^*(\eta_\beta)_{\alpha_!}}&&
\alpha^*\beta^*\beta_!\alpha_!
}
$$
\end{lemma}

The following condition on a representation is fundamental in defining the category of cartesian modules:

\begin{definition}\label{def.flat.rep}
A representation $R$ of a small category $\C$ is {\em right} (resp., {\em left}) {\em flat} if, for any $\alpha\in\C$, the functor $R_\alpha$ is right (resp., left) flat.  
\end{definition}
The three representations described in Example \ref{geo_ex} are all (left and right) flat (see \cite{SergioQcoh,Sergio_Enochs_Rel_homo}).

\subsection{The category of modules}\label{sub.sec.cat.mod}

\begin{definition}\label{def_mod}
Let $R\colon \C\to \Add$ be a representation of the small category $\C$. A {\em right $R$-module} $M$ consists of the following data:
\begin{enumerate}[\rm --]
\item for all $c\in \Ob\C$, a right $R_c$-module $M_c\colon R_c^{op}\to \Ab$;
\item for any morphism $\alpha\colon c\to d$ in $\C$, a homomorphism $M_\alpha\colon M_c\to \alpha^*M_d$.
\end{enumerate}
Furthermore, we suppose that the following axioms hold:
\begin{enumerate}[\rm (Mod.1)]
\item given two morphisms $(\alpha\colon c\to d),\, (\beta\colon d\to e)$ in $\C$, the following diagram commutes:
$$\xymatrix{
M_c\ar@/_10pt/[rrrrd]_{M_{\beta\alpha}}\ar[rr]^{M_\alpha}&&\alpha^*M_d\ar[rr]^{\alpha^*M_{\beta}}&&\alpha^*\beta^*M_e\\
&&&&(\beta\alpha)^*M_e\ar[u]_{\mu_{\alpha,\beta}*\id_{M_e}}
}$$
\item $(\delta_c*\id_{M_c})\circ M_{\id_c}=\id_{M_c}$, for all $c\in \Ob\C$;
\end{enumerate}
where $\mu_{\alpha,\beta}*\id_{M_e}$ and $\delta_c*\id_{M_c}$ are horizontal pastings, as described before Lemma \ref{change_of_two}. 
Given two right $R$-modules $M$ and $N$, a morphism $\phi\colon M\to N$ consists of a family of morphisms  $\phi_c\colon M_c\to N_c$ (one for any $c\in\Ob\C$) in $(R(c)^{op},\Ab)$ such that the following square commutes for any morphism $\alpha\colon c\to d$ in $\C$:
$$\xymatrix{
M_c\ar[rr]^{\phi_c}\ar[d]_{M_\alpha}&&N_c\ar[d]^{N_\alpha}\\
\alpha^*(M_d)\ar[rr]^{\alpha^*(\phi_d)}&&\alpha^{*}(N_d)
}$$
\end{definition}

We denote by $\mod R$ the category of right $R$-modules. Notice that one can define analogously the category $\lmod {R}$ of left $R$-modules. When no confusion is possible we will say module to mean right $R$-module. Notice that, if $R$ is a strict representation, the axioms (Mod.1) and (Mod.2) for a right $R$-module $M$, boil down to the conditions: $M_{\beta\alpha}=\alpha^*M_{\beta}\circ M_{\alpha}$ and $M_{\id_c}=\id_{M_c}$.

\begin{definition}
Let $R$ be a representation of a small category $\C$ and let $M\in\mod R$. Then we let
$$|M|=\sum_{c\in\Ob \C}|M_c|\,.$$
\end{definition}

\begin{lemma}
Let $\C$ be a small category and let $R$ be a representation of $\C$. Then $\mod R$ is a bicomplete Abelian category. Furthermore, $\mod R$ is $(Ab.4^*)$ and $(Ab.5)$.
\end{lemma}
\begin{proof}
Let us start showing that $\mod R$ is a bicomplete Abelian category. Indeed, $\mod R$ has a zero-object (the object with $0$ in each component). Arbitrary products and coproducts can be taken componentwise  (use the fact that the functors $\alpha^*$ commute with both products and coproducts). Similarly, kernels and cokernels can be taken componentwise (use the fact that the functors $\alpha^*$ are exact), thus it is also clear that the canonical morphism between image and coimage is an isomorphism, as it is an isomorphism in each component. By this description of kernels, cokernels, products and coproducts, it follows that products and colimits are exact in $\mod R$, being exact componentwise.
\end{proof}

The subtle point in proving that the category of modules is a Grothendieck category, is to show that it has a generator.

The lemma below generalizes the following well-known fact in module theory: let $R$ and $S$ be rings and let $\phi$ be a ring homomorphism, given a right $R$-module $M$ and an element $x\in M\otimes_R S$, there are elements $m_1,\dots, m_n\in M$ and $s_1,\dots, s_n\in S$ such that $x=\sum_{i=1}^nm_i\otimes s_i$.

\begin{lemma}
Let $R$ and $S$ be small preadditive categories, let $\phi\colon R\to S$ be a functor and let  $M\in(R^{op},\Ab)$. Then, 
\begin{enumerate}[\rm (1)]
\item  there exists a family $\{a_i: i\in I\}$ of objects of $R$, a family of sets $\{A_i:i\in I\}$ and an epimorphism in $(S^{op},\Ab)$
$$\bigoplus_{i\in I}(H_{\phi(a_i)})^{(A_i)}\to \phi_!(M)\,;$$
\item given $b\in \Ob S$ and a $b$-element $f\colon H_b\to \phi_!(M)$ there exists a finite set $\{a_1,\dots, a_n\}$ of objects of $R$ and a morphism $f':\bigoplus_{i=1}^nH_{a_i}\to M$ such that $f$ factors through $\phi_!(f')$, that is
$$\xymatrix{
&&\bigoplus_{i=1}^nH_{\phi(a_i)}\ar[d]|{\phi_!(f')}\\
H_b\ar@/_-10pt/@{.>}[rru]|\exists\ar[rr]|{f}&&\phi_!M
}$$
\end{enumerate}
\end{lemma}
\begin{proof}
(1) Choose a family $\{a_i:i\in I\}$ of objects of $R$, a family of sets $\{A_i:i\in I\}$ and an epimorphism in $(R^{op},\Ab)$
$$\bigoplus_{i\in I}(H_{a_i})^{(A_i)}\to M\,.$$
Since $\phi_!$ is a left adjoint, it is cocontinuous and right exact, thus the following morphism is an epimorphism
$$\bigoplus_{i\in I}(\phi_!H_{a_i})^{(A_i)}\to \phi_!(M)\,.$$
To conclude it is enough to show that $\phi_!(H_{a_i})\cong H_{\phi(a_i)}$ for all $i\in I$. But in fact, for all $b\in \Ob S$, 
$$(\phi_!(H_{a_i}))(b)= \hom(-,a_i)\otimes\hom(b,\phi(-))\cong \hom(b,\phi(a_i))=H_{\phi(a_i)}(b)\,.$$
(2) follows by part (1), by the fact that $H_b$ is a projective object (so that $f$ can be lifted along the epimorphism constructed in (1)) and since it is finitely presented (so that our lifting factors through a finite sub-coproduct).  
\end{proof}

Let $R$ be a small category, in what follows we denote by $\mathrm{Mor}(R)$ the set of all morphisms in $R$. 

\begin{corollary}\label{coro_pre_gen}
Let $R$ and $S$ be small preadditive categories, let $\kappa=\max\{|\N|,|\mathrm{Mor}(R)|, |\mathrm{Mor}(S)|\}$ and let $\phi\colon R\to S$ be a right flat functor. Let $M\in(R^{op},\Ab)$ and let $X$ be a set of elements of $\phi_!(M)$ such that $|X|\leq \kappa$. Then, there exists $M'\leq M$ such that 
\begin{enumerate}[\rm (1)]
\item $|M'|\leq \kappa$;
\item for all $x\in X$, $x$ is an element of $\phi_!(M')(\leq \phi_!(M))$;
\item $M'$ is pure in $M$ and $\phi_!(M')$ is pure in $\phi_!(M)$.
\end{enumerate}
\end{corollary}
\begin{proof}
For any $x\in X$, there is a finite family $\{a_1(x),\dots, a_{n_x}(x)\}$ of objects of $R$ and a morphism $f_x\colon \bigoplus_{i=1}^{n_x}H_{a_i(x)}\to M$ such that $x$ factors through $\phi_!(f_x)$. Consider the disjoint union
$$B=\bigsqcup_{x\in X}\{a_1(x),\dots, a_{n_x}(x)\}\,,$$
let $f=\bigoplus_{x\in X} f_x\colon \bigoplus_{x\in X}(\bigoplus_{i=1}^{n_x}H_{a_i(x)})\to M$, and set $N=\Im(f)\leq M$. Consider the following observations:
\begin{enumerate}[\rm --]
\item $|H_r|\leq |\mathrm{Mor(R)}|\leq \kappa$ for all $r\in \Ob R$;
\item $|B|\leq \max\{|\N|,|X|\}\leq \kappa$ (as there is an obvious finite-to-one map $B\to X$).
\end{enumerate}
By Lemma \ref{card_lemma}, $|N|\leq \kappa$. Using Theorem \ref{card_purification}, we can take $M'\leq M$ such that $N\leq M'$, $M'$ is pure and $|M'|\leq \kappa$. To conclude it is enough to show that $\phi_!(M')$ is pure in $\phi_!(M)$. Indeed, given $K\in (S,\Ab)$, we have to show that $0\to \phi_!(M')\otimes_S K\to \phi_!(M)\otimes_S K$ is exact. But in fact, 
$$\phi_!(M)\otimes_S K=(M\otimes_R\hom(-,\phi(-)))\otimes_S K\cong M\otimes_R(\hom(-,\phi(-))\otimes_S K)\,,$$
and the same holds for $M'$, so that, letting $K'=\hom(-,\phi(-))\otimes_S K\in (R,\Ab)$, we obtain that $0\to \phi_!(M')\otimes_S K\to \phi_!(M)\otimes_S K$ is exact if and only if $0\to M'\otimes_RK'\to M\otimes_RK'$
is exact, which is true by the purity of $M'$ in $M$.
\end{proof}

%Before proceeding to do so, we want to simplify a little our task. Indeed, let $\R$  be a representation of a small category $\C$. Since left adjoints are determined up to natural isomorphism and since in the definition of modules, for any given $\alpha\in \C$, only the right adjoint $\alpha^*$ is mentioned, it does not seem that the left adjoint $\alpha_!$ plays a real role. This is made precise in the following lemma.
%
%\begin{lemma}
%Let $\R$ and $\R'$ be two representations of the same small category $\C$ and suppose that
%\begin{enumerate}[\rm (1)]
%\item $R=R'$, that is, $\R$ and $\R'$ have the same underlying functors $\C\to \Add$;
%\item $\alpha^*=(\alpha')^*$, for all $\alpha\in\C$;
%\item $\mu(\beta,\alpha)=\mu'(\beta,\alpha)$, for all composable $\alpha,\beta\in\C$. 
%\end{enumerate}
%Then there is a natural equivalence of categories $\mod \R\cong \mod \R'$.
%\end{lemma}
%
%By the above lemma we are allowed to assume that $\alpha_!=R(\alpha)_!$, without any loss in the generality of our results. Notice that, under this assumption,
%$$\beta_!\alpha_!=(\beta\alpha)_!\,.$$
%Let us go back to our proof that $\mod \R$ is a Grothendieck category. 
%
%Now, for all $(\alpha:c\to d)\in \mor \C$, the structural map $M_\alpha:M_c\to \alpha^*(M_d)$ consists of a family of homomorphisms of abelian groups $(M_\alpha)_a:(M_c)_a\to (\alpha^*(M_d))_a$, for all $a\in \Ob\C$. Notice also that, just by definition, $(\alpha^*(M_d))_a=(M_d)_{R(\alpha)(a)}$, so that $M_\alpha$ consists of a family of morphisms $(M_\alpha)_a:(M_c)_a\to (M_d)_{R(\alpha)(a)}$ in $\Ab$. 

\begin{definition}
Let $R$ be a representation of a small category $\C$ and let $M$ be a right $R$-module. Given $\alpha\in\C$, $c\in\Ob \C$ and $a\in \Ob R_c$, the morphism $M_\alpha$ induces a homomorphism of Abelian groups
$$M_\alpha\colon \Hom_{(R_c^{op},\Ab)}(H_a, M_c)\to \Hom_{(R_d^{op},\Ab)}(H_{R_\alpha(a)}, M_d)$$
which sends $f\in\Hom_{(R_c^{op},\Ab)}(H_a, M_c)$ to the adjoint morphism  of $M_\alpha\circ f$ with respect to the adjunction $(\alpha_!,\alpha^*)$ (we are using implicitly the natural isomorphism $\alpha_!H_a\cong H_{R_\alpha(a)}$).
\end{definition}

%If we have two composable arrows $\alpha:c\to d$ and $\beta:d\to e$ in $\mor \C$ and $f\in M_c(a)$, for some $a\in\Ob R(c)$, then 
%\begin{equation}\label{compo_element}
%M_{\beta\alpha}(f)=M_{\beta}(M_{\alpha}(f))\in (M_e)_{R(\beta\alpha)(a)}\,.
%\end{equation}
The following lemma is the technical ingredient needed to prove that $\mod R$ is a Grothendieck category, which will be concluded in the successive corollary.

\begin{lemma}\label{tech_gen}
Let $R$ be a representation of a small category $\C$, and let 
$$\kappa=\sup\{|\N|,|\mor \C|, |\mor R_c|:c\in \Ob\C\}\,.$$ 
Let also $M\in\mod R$, $f\in M_c$ and let $N\leq M$ be a submodule such that $|N|\leq \kappa$. Then, there exists a submodule $N'$ of $M$ such that $|N'|\leq\kappa$, $N\leq N'$ and $f\in N'_c$.
\end{lemma}
\begin{proof}
Let $N'_c$ be the $R_c$-submodule of $M_c$ generated by $N_c$ and $\{M_\alpha(f): \alpha\in\End_\C(c)\}$. Furthermore, we let $N'_d$ be the $R_d$-submodule of $M_d$ generated by $N_d$ and $\bigcup_{\alpha\in\hom_\C(c,d)}M_\alpha(N_c')$ (where $M_\alpha(N_c')=\{M_\alpha(f):f\in N_c'\}$ is a set of elements of $N_d'$).  It is not difficult to check that $|N'_d|\leq \kappa$, for all $d\in \Ob\C$.

Let now $(\beta\colon d\to e)\in \C$, we want to show that $M_\beta\colon M_{d}\to \beta^*M_e$ restricts to a morphism $N'_\beta\colon N'_d\to \beta^*N_e$. By Lemma \ref{restriction_map} we should check that, for all $a\in\Ob R_d$ and all $f\colon H_a\to N'_d\to M_d$, there exists $g\colon H_a\to \beta^*N'_e\to \beta^*M_e$ such that $M_\beta\circ f=g$. Indeed, by definition of $N_e'$, the morphism $M_\beta(f)\colon H_{R_\beta(a)}\to M_e$ factors through the inclusion $N_e'\to M_e$, the desired map $g\colon H_a\to \beta^*N_e'\to \beta^*M_e$ is the adjoint morphism to $M_\beta(f)$ along the adjunction $(\beta_!,\beta^*)$.
\end{proof}

\begin{corollary}
Let $R$ be a representation of a small category $\C$, then $\mod R$ has a generator.
\end{corollary}
\begin{proof}
Notice first that, by Lemma \ref{tech_gen}, any $R$-module $M$ is the sum of its submodules $N$ such that $|N|\leq \kappa$ (for $\kappa$ as in the statement of the lemma). Then, if we take a family of representatives $\F$ of all the $R$-modules $N$ such that $|N|\leq \kappa$ up to isomorphism, then $\F$ is a family of generators. Since $\mod R$ is cocomplete, $\bigoplus\F$ is the generator we are looking for.
\end{proof}

Remember that a poset is a small category $\C$ such that $|\hom_\C(c,d)|\leq 1$ for all $c,d\in \Ob\C$. In the particular case when $\C$ is a poset we can prove that, given a representation $R$ of $\C$, $\mod R$ is not only Grothendieck, but it has also a projective generator (see Corollary \ref{generatori_proj}). 

\begin{definition}
Let $\C$ be a poset, let $c\in \Ob \C$ and let $R$ be a representation of $\C$. The {\em extension by zero functor at $c$}, $\ext_c\colon (R_c^{op},\Ab)\to \mod R$ is defined by sending an $R_c$-module $M$ to $\ext_c(M)$, where
$$\ext_c(M)_d=\begin{cases}
\alpha_!M &\text{if $\emptyset\neq \hom_{\C}(c,d)\ni\alpha$;}\\
0&\text{otherwise.} 
\end{cases}$$
and where the structural map $\ext_c(M)_\beta$ is given by the following composition:
$$\xymatrix{\ext_c(M)_\beta\colon (\alpha_d)_!M\ar[rr]^{(\varepsilon_\beta)_{(\alpha_d)_!M}}&&\beta^*\beta_!(\alpha_d)_!M\ar[rr]^{\beta^*\tau_{\beta,\alpha_d}}&&  \beta^*(\alpha_e)_!M}\,,$$ 
for all $\beta\in\hom_\C(d,e)$,  $\alpha_d\in \hom_\C(c,d)$ and $\alpha_e\in \hom_\C(c,e)$.
%, where:
%\begin{enumerate}[\rm --]
%\item the map $(\alpha_d)_!M\to\beta^*\beta_!(\alpha_d)_!M$ is the component at $(\alpha_d)_!M$ of the unit of the adjunction $(\beta_!,\beta^*)$ (as described in the first part of Section \ref{change_of_base_sec});
%\item the map $\beta^*\beta_!(\alpha_d)_!M\to \beta^*(\beta\alpha_d)_!M$ is induced by the natural isomorphism described in Lemma \ref{functoriality};
%\item the map $\beta^*(\beta\alpha_d)_!M\to \beta^*(\alpha_e)_!M$ is induced by $\mu_{\beta,\alpha_d}:R_\beta R_{\alpha_d}\to R_{\beta\alpha_d}=R_{\alpha_e}$.
%\end{enumerate}
More explicitly, the component at a given $b\in \Ob R_d$ of the above composition sends a standard generator $m\otimes f\in ((\alpha_d)_!M)_b=\bigoplus_{a\in \Ob R_c}M_a\otimes_\Z \hom_{R_d}(b,R_{\alpha_d}(a))/T$ to $m\otimes ((\mu_{\beta,\alpha_d})_a\circ R_\beta(f))\in \beta^*(\alpha_e)_!M$.

\smallskip
Furthermore, given a homomorphism $\phi\colon M\to N$ in $(R_c^{op},\Ab)$ and $d\in \Ob\C$, we let 
$$\ext_c(\phi)_d=\begin{cases}
\alpha_!\phi&\text{if $\alpha\in\hom_\C(c,d)\neq\emptyset$;}\\
0&\text{otherwise.}
\end{cases}$$ 
One defines analogously the functor $\ext^*_c\colon (R_c,\Ab)\to \lmod R$.
\end{definition}

To verify that $\ext_c(M)$ is a module one has to check the axioms (Mod.1) and (Mod.2). Indeed, the axiom (Mod.1) consists in verifying the commutativity of the following diagram:
$$\xymatrix{
\ext_c(M)_d\ar@/_10pt/[rrrrd]_{\ext_c(M)_{\gamma\beta}}\ar[rr]^{\ext_c(M)_\beta}&&\beta^*\ext_c(M)_e\ar[rr]^{\beta^*\ext_c(M)_\gamma}&&\beta^*\gamma^*\ext_c(M)_f\\
&&&&(\gamma\beta)^*\ext_c(M)_f\ar[u]_{\mu_{\gamma,\beta}*\id_{\ext_c(M)_f}}
}$$
where $\alpha_d\colon c\to d$, $\alpha_e\colon c\to e$, $\alpha_f\colon c\to f$, $\beta\colon d\to e$, and $\gamma\colon e\to f$. Let us check commutativity on an element $m\otimes h\in ((\alpha_d)_!M)_b$, (where $m\in M_a$ and $h\colon b\to R_{\alpha_d}(a)$, for some $a\in\Ob R_c$):
$$\xymatrix{
m\otimes h\ar@{|->}@/_10pt/[rrrrdd]\ar@{|->}[rr]&&m\otimes ((\mu_{\beta,\alpha_d})_a\circ R_\beta(h))\ar@{|->}[rr]&&m\otimes ((\mu_{\gamma,\alpha_e})_a\circ R_\gamma(\mu_{\beta,\alpha_d})_a\circ R_\gamma R_\beta(h))\\
&&&&m\otimes ((\mu_{\gamma\beta,\alpha_d})_a\circ R_{\gamma\beta}(h)\circ (\mu_{\gamma,\beta})_{b})\ar@{=}[u]\\
&&&&m\otimes ((\mu_{\gamma\beta,\alpha_d})_a\circ R_{\gamma\beta}(h))\ar@{|->}[u]
}$$
where equality holds by the commutativity of the following diagram:
$$\xymatrix{
R_\gamma R_\beta(b)\ar[rr]^{R_\gamma R_\beta(h)}\ar[d]_{(\mu_{\gamma,\beta})_b}&&R_\gamma R_\beta R_{\alpha_{d}}(a)\ar[rr]^{R_\gamma(\mu_{\beta,\alpha_d})_a}\ar[d]|{(\mu_{\gamma,\beta})_{R_{\alpha_d}(a)}}&&R_\gamma R_{\beta\alpha_d}(a)\ar[d]^{(\mu_{\gamma,\beta\alpha_d})_a}\\
R_{\gamma\beta}(b)\ar[rr]^{R_{\gamma\beta}(h)}&&R_{\gamma \beta}R_{\alpha_d}(a)\ar[rr]^{(\mu_{\gamma\beta,\alpha_d})_a}&&R_{\gamma\beta\alpha_d}(a)
}$$
where the left-hand square commutes since $\mu_{\gamma,\beta}$ is a natural transformation, while the right-hand square commutes by the axiom (Rep.1).

\medskip
Let us now verify (Mod.2). Indeed, given $\alpha_d\colon c\to d$, we have to show that, for all $b\in\Ob R_c$, 
$$((\delta_d*\id_{\ext_c(M)_d})\circ \ext_c(M)_{\id_d})_b=(\id_{\ext_c(M)_d})_b\,.$$
This amounts to say that, given $m\otimes h\in ((\alpha_d)_!M)_b$ (with $m\in M_a$ and $h\colon a\to R_{\alpha_d}(b)$ for some $a\in\Ob R_d$),
$$m\otimes ((\mu_{\id_d,\alpha_d})_b\circ R_{\id_d}(h)\circ (\delta_d)_a)=m\otimes h$$
and this equality follows by the commutativity of the following diagram:
$$\xymatrix{
a\ar[d]_{(\delta_d)_a}\ar[rr]^h&& R_{\alpha_d}(b)\ar[d]|{(\delta_{d})_{R_{\alpha_d}(b)}}\ar@/_-10pt/[drr]^{\id_{R_{\alpha_d}}}\\
R_{\id_d}a\ar[rr]^{R_{\id_d}(h)} &&R_{\id_d}R_{\alpha_d}(b)\ar[rr]^{(\mu_{\id_d,\alpha_d})_b}&&R_{\alpha_d}(b)
}$$
where the left-hand square commutes since $\delta_d$ is a natural transformation, while the triangle commutes by the axiom (Rep.2).

% it is enough to remember the formula for the unit of the composition of two adjunctions. Indeed, given $c\overset{\alpha}{\to}d\overset{\beta}{\to}e\overset{\gamma}{\to}f$ in $\C$ and letting $\hom_\C(c,e)=\{\alpha_e\}$, the component at $\alpha_!M$ of the unit of $((\gamma\beta)_!,(\gamma\beta)^*)\cong (\gamma_!\beta_!,\beta^*\gamma^*)$ is
%$$\beta^*(\eta_{\gamma})_{(\alpha_e)_!M}\circ(\eta_\beta)_{\alpha_!M}\,,$$
%where $\eta_\beta$ and $\eta_\gamma$ are the units respectively of $(\beta_!,\beta^*)$ and $(\gamma_!,\gamma^*)$.
%
\begin{remark}\label{rem_ex}
Let $\C$ be a poset and let $R$ be a right (resp., left) flat representation of $\C$. Then, just by definition, $\ext_c$ (resp., $\ext_c^*$) is an exact functor.
\end{remark}

Let $\C$ be a small category and let $R$ be a representation of $\C$. Then there is a canonical functor 
$$\ev_c\colon \mod R\to (R_c^{op},\Ab)$$
called {\em evaluation at $c$}, that sends a right $R$-module $M$ to $M_c$ and a morphism $\phi$ to its $c$-component $\phi_c$. In the following proposition we are going to show that, in case $\C$ is a poset, $\ext_c$ is left adjoint to $\ev_c$. 

\begin{proposition}\label{aggiunto_ev}
Let $\C$ be a poset and let $R$ be a representation of $\C$. Then, for all $M\in\mod R$ and $N\in (R_c^{op},\Ab)$ there is a natural isomorphism
$$\hom_{\mod R}(\ext_cN,M)\cong \hom_{(R_c^{op},\Ab)}(N,M_{c})\,.$$ 
\end{proposition}
\begin{proof}
Define a map $\Phi\colon \hom_{\mod R}(\ext_cN,M)\to \hom_{(R_c^{op},\Ab)}(N,M_{c})$ such that 
$$\Phi(\phi\colon \ext_cN\to M)=(\phi_c\colon N\to M_c)\,.$$
It is easily seen that $\Phi$ is a group homomorphism. Let us show that $\Phi$ is bijective.\\
{\em $\Phi$ is injective.} Take $\phi\in\ker(\Phi)$, that is, $\phi_c=0$. Let $d\in \Ob \C$, we want to show that $\phi_d=0$. This is clear in case $\hom_\C(c,d)=\emptyset$, so suppose $\alpha\in \hom_\C(c,d)\neq\emptyset$. Notice that there is a commutative diagram
$$
\xymatrix{
\alpha_!N\ar[rr]^{\alpha_!(\phi_c)=\alpha_!(0)=0}\ar@/_4.5pc/[dd]_\id\ar[d]_{\alpha_!((\eta_\alpha)_N)}&&\alpha_!M_c\ar[d]^{\alpha_!(M_\alpha)}\\
\alpha_!\alpha^*\alpha_!N\ar[rr]^{\alpha_!\alpha^*(\phi_d)}\ar[d]_{\varepsilon_{\alpha_!(N)}}&&\alpha_!\alpha^*M_d\ar[d]^{\varepsilon_{M_d}}\\
\alpha_!N\ar[rr]^{\phi_d}&&M_d\\
}
$$
where $\varepsilon$ and $\eta$ are respectively the counit and the unit of the adjunction $(\alpha_!,\alpha^*)$. This implies that $\phi_d=0$, so that $\phi=0$, as desired. \\
{\em $\Phi$ is surjective.} Let $\phi\colon N\to M_c$. We are going to define a morphism $\psi\colon \ext_c(N)\to M$ such that $\psi_c=\phi$:
$$\psi_d=\begin{cases}
(\varepsilon_{\alpha_d})_{M_d}\circ(\alpha_d)_!(M_{\alpha_d}\circ \phi)&\text{if $\emptyset\neq\hom_{\C}(c,d)\ni\alpha_d$;}\\
0&\text{otherwise;}
\end{cases}$$
where $\varepsilon_{\alpha_d}$ is the counit of the adjunction $((\alpha_d)_!,(\alpha_d)^*)$. Let us show that this defines a homomorphism of right $R$-modules. Indeed, consider the following diagram:
$$\includegraphics[height=550pt]{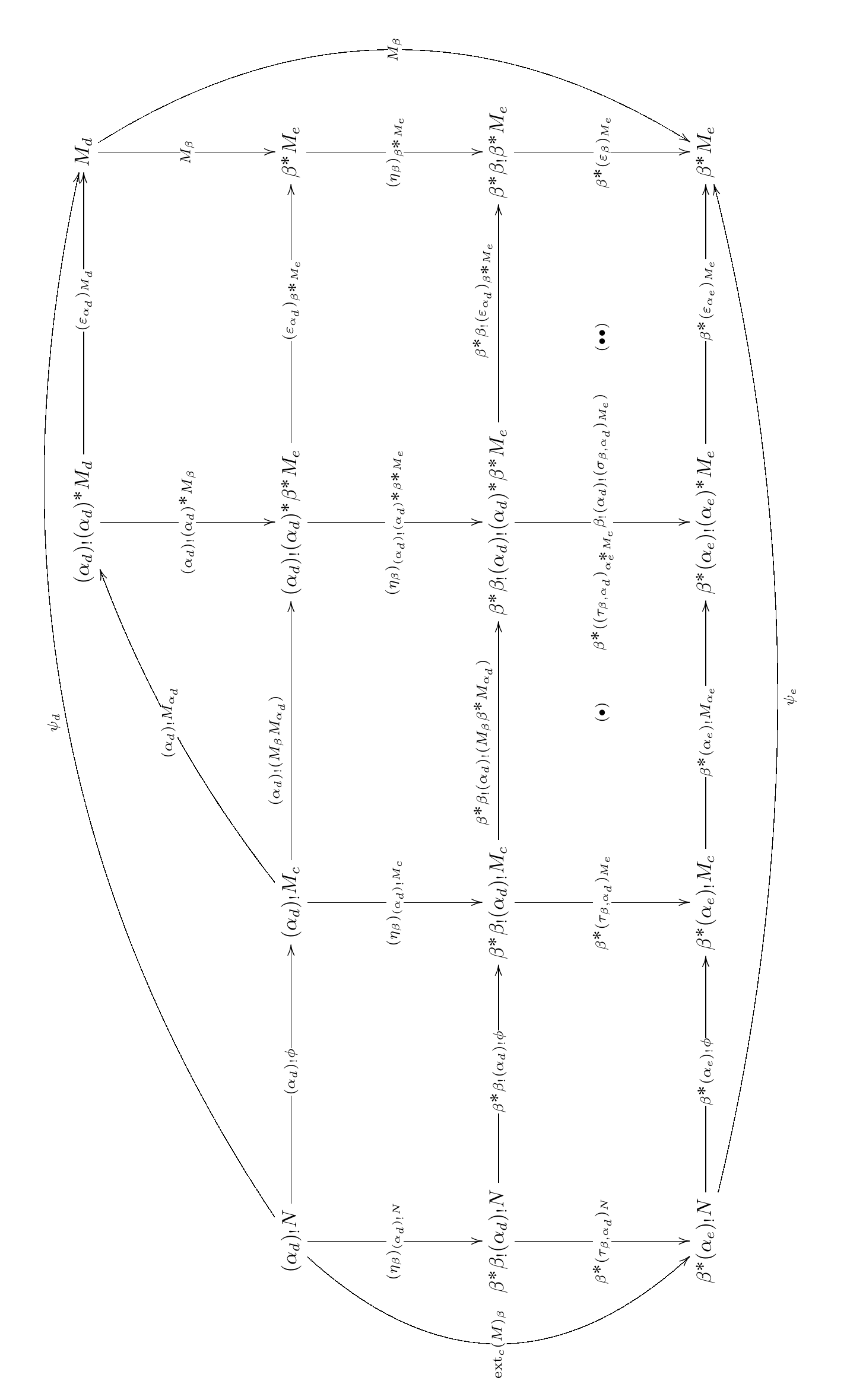}$$
%
%$$\xymatrix{
%&&
%&&
%(\alpha_d)_!(\alpha_d)^*M_d\ar[d]|{(\alpha_d)_!(\alpha_d)^*M_\beta}\ar[rr]|(0.6){(\varepsilon_{\alpha_d})_{M_d}}&&
%M_d\ar[d]|{M_\beta}\ar@/_-40pt/[ddd]|{M_\beta}
%\\
%(\alpha_d)_!N\ar@/_-30pt/[urrrrrr]^{\psi_d}\ar@/_46pt/[dd]|{\ext_c(M)_\beta}\ar[rr]|{(\alpha_d)_!\phi}\ar[d]|{(\eta_{\beta})_{(\alpha_d)_!N}}&&
%(\alpha_d)_!M_c\ar[d]|{(\eta_\beta)_{(\alpha_d)_!M_c}}\ar[rr]^(0.4){(\alpha_d)_!(M_\beta M_{\alpha_d})}\ar@/_-10pt/[rru]|{(\alpha_d)_!M_{\alpha_d}}&&
%(\alpha_d)_!(\alpha_d)^*\beta^*M_e\ar[d]|{(\eta_\beta)_{(\alpha_d)_!(\alpha_d)^*\beta^*M_e}}\ar[rr]|(0.6){(\varepsilon_{\alpha_d})_{\beta^*M_e}}&&
%\beta^*M_e\ar[d]|{(\eta_\beta)_{\beta^*M_e}}
%\\
%\beta^*\beta_!(\alpha_d)_!N\ar[rr]|{\beta^*\beta_!(\alpha_d)_!\phi}\ar[d]|{\beta^*(\tau_{\beta,\alpha_d})_N}&&
%\beta^*\beta_!(\alpha_d)_!M_c\ar[d]|{\beta^*(\tau_{\beta,\alpha_d})_{M_e}}\ar[rr]^(0.42){\beta^*\beta_!(\alpha_d)_!(M_\beta \beta^*M_{\alpha_d})}\ar@{}[rrd]|(0.4){(\bullet)}&&
%\beta^*\beta_!(\alpha_d)_!(\alpha_d)^*\beta^*M_e\ar[rr]^(0.6){\beta^*\beta_!(\varepsilon_{\alpha_d})_{\beta^*M_e}}\ar[d]|{\beta^*((\tau_{\beta,\alpha_d})_{\alpha_e^*M_e}\beta_!(\alpha_d)_!(\sigma_{\beta,\alpha_d})_{M_e})}\ar@{}[rrd]|(0.6){(\bullet\bullet)}&&
%\beta^*\beta_!\beta^*M_e\ar[d]|{\beta^*(\varepsilon_\beta)_{M_e}}
%\\
%\beta^*(\alpha_e)_!N\ar@/_25pt/[rrrrrr]_{\psi_e}\ar[rr]|{\beta^*(\alpha_e)_!\phi}&&
%\beta^*(\alpha_e)_!M_c\ar[rr]|(0.45){\beta^*(\alpha_e)_!M_{\alpha_e}}&&
%\beta^*(\alpha_e)_!(\alpha_e)^*M_e\ar[rr]|(0.6){\beta^*(\varepsilon_{\alpha_e})_{M_e}}&&
%\beta^*M_e
%}$$
where everything commutes either by definition, by the naturality of units and counits or by the unit-counit equations, apart from the squares marked by $(\bullet)$ and $(\bullet \bullet)$. But in fact, $(\bullet)$ commutes by the definition of module applied to $M$, while the commutativity of $(\bullet\bullet)$ is proved in Lemma \ref{unit_of_two}.
\end{proof}

As a consequence we obtain the following

\begin{corollary}\label{generatori_proj}
Let $\C$ be a poset, let $c\in \Ob \C$ and let $R$ be a representation of $\C$. For all $a\in \Ob R_c$, 
$$\hom_{\mod R}(\ext_c(H_a),M)\cong M_c(a)\,.$$ 
In particular, $\{H_a^c:c\in \Ob\C,\ a\in \Ob R_c\}$ is a set of projective generators of $\mod R$. 
\end{corollary}
The above results can be summarized as follows:

\begin{theorem}\label{theor.mod.Grot}
Let $\C$ be a small category and let $R$ be a representation of $\C$. Then $\mod R$ is an $(Ab.4^*)$ Grothendieck category. If furthermore $\C$ is a poset, then $\mod R$ has a projective generator. 
\end{theorem}

Let us underline another consequence of Proposition \ref{aggiunto_ev}:

\begin{corollary}
Let $\C$ be a poset and let $R$ be a representation of $\C$. If $E$ is an injective object in $\mod R$, then $E_c$ is injective in $(R_c^{op},\Ab)$ for all $c\in \Ob\C$.
\end{corollary}

\subsection{Flat modules}\label{sub.sec.flat}

In this subsection we introduce a tensor product for modules over a representation and then we study the induced notion of flatness. Let us start with the following

\begin{definition}\label{def_flat_module_rep}
Let $R\colon \C\to \Add$ be a representation of the small category $\C$, the {\em tensor product} over $R$ is a functor
$$-\otimes_R-\colon \mod R\times \lmod {R}\to \Ab$$
such that 
\begin{enumerate}[\rm (1)]
\item $M\otimes_R N=\bigoplus_{c\in \Ob \C}M_c\otimes_{R_c} N_c$, for $M\in \mod R$ and $N\in\lmod {R}$;
\item given two morphisms $\phi\colon M_1\to M_2$, $\psi\colon N_1\to N_2$, respectively of right $R$-modules and left ${R}$-modules, we define 
$$\phi\otimes_R \psi=\bigoplus_{c\in\Ob\C}\phi_c\otimes_{R_c}\psi_c\colon M_1\otimes_R N_1\to M_2\otimes_R N_2\,.$$
\end{enumerate}
A right $R$-module $M$ is said to be {\em flat} if the functor $(M\otimes_R-)\colon \lmod {R}\to \Ab$ is exact.
\end{definition}

Let us remark that the map $\phi\otimes_R \psi$ in the above definition, when written in matricial form, gives a $|\Ob\C|\times |\Ob\C|$ diagonal matrix, whose diagonal entry corresponding to $c\in \Ob\C$ is  
$$\phi_c\otimes_{R_c}\psi_c\colon (M_1)_c\otimes_{R_c}(M_2)_c\to (N_1)_c\otimes_{R_c}(N_2)_c\,.$$

\begin{proposition}\label{flat_obj_loc}
Let $R$ be a representation of the small category $\C$ and let $M$ be a right $R$-module. If $M_c$ is flat in $(R_c^{op},\Ab)$ for all $c\in \Ob \C$, then $M$ is flat. 

If furthermore $\C$ is a poset and $R$ is left flat, then also the converse is true, that is, $M$ is flat in $\mod R$ if and only if $M_c$ is flat in $(R_c^{op},\Ab)$, for all $c\in \Ob\C$.  
\end{proposition}
\begin{proof}
Suppose first that $M_c$ is flat in $(R_c^{op},\Ab)$ for all $c\in \Ob \C$ and let 
$$0\to N_1\to N_2\to N_3\to 0$$
be a short exact sequence in $\lmod R$. We should prove that $0\to M\otimes_R N_1\to M\otimes_R  N_2\to M\otimes_R  N_3\to 0$ is exact in $\Ab$, but this is clear since the coproduct is an exact functor and each $M_c\otimes_{R_c}-$ is an exact functor. 

Suppose now that $\C$ is a poset, that $R$ is left flat and that $M$ is flat in $\mod R$. Choose $c\in\Ob \C$ and let us show that $M_c$ is flat in $(R(c)^{op},\Ab)$. Indeed, let 
$$0\to A_1\to A_2\to A_3\to 0$$
be a short exact sequence in $(R_c,\Ab)$. Then, 
$$0\to \ext^*_c(A_1)\to \ext^*_c(A_2)\to \ext^*_c(A_3)\to 0$$
is a short exact sequence in $\lmod R$, and so, by the flatness of $M$, the sequence
$$0\to M\otimes_R\ext^*_c(A_1)\to M\otimes_R\ext^*_c(A_2)\to M\otimes_R\ext^*_c(A_3)\to 0$$
is exact in $\Ab$. Since the morphisms in the above short exact sequence are diagonal matrices, it follows that the following sequence is short exact
$$0\to (\pi_1)_c(M\otimes_R\ext^*_c(A_1))\to (\pi_2)_c(M\otimes_R\ext^*_c(A_2))\to (\pi_3)_c(M\otimes_R\ext^*_c(A_3))\to 0\,,$$
where $(\pi_i)_c\colon \bigoplus_{d\in\Ob\C}M_d\otimes_{R_d}\ext^*_c(A_i)_d\to M_c\otimes_{R_c}\ext^*_c(A_i)_c=M_c\otimes_{R_c}A_i$  (for $i=1,2,3$) is the canonical projection from the coproduct in $\Ab$.
Thus, we obtained that the following sequence is short exact
$$0\to M_c\otimes_{R_c}A_1\to M_c\otimes_{R_c}A_2\to M_c\otimes_{R_c}A_3\to 0\,,$$
proving that $M_c$ is flat.
\end{proof}

\subsection{The category of cartesian modules}\label{sub.sec.cat.mod.cart}

\begin{definition}\label{def_cart_mod}
Given a representation $R\colon \C\to \Add$ of a small category $\C$, the category of {\em cartesian} modules $\cmod R$ is the full subcategory of $\mod R$ whose objects are the modules $M$ such that, for any given morphism $\alpha\colon c\to d$ in $\C$, the adjoint morphism $(M_\alpha)_!\colon \alpha_!(M_c)\to M_d$ to the structural morphism $M_\alpha\colon M_c\to \alpha^*(M_d)$ is an isomorphism. 
\end{definition}

The main result of this subsection is to show that the category of cartesian modules over a flat representation is Grothendieck. 

\begin{theorem}\label{cart_groth}
Let $\C$ be a small category and consider a right flat representation $R\colon \C\to \Add$. Then, the category $\cmod R$ is Grothendieck. 
\end{theorem}
\begin{proof}
We start showing that $\cmod R$ is an Abelian subcategory of $\mod R$. In fact, it is clear that the $0$-module is cartesian. Let $\phi\colon M\to N$ be a morphism between two cartesian modules, one has to show that $\ker(\phi)$ and $\coker(\phi)$ are again cartesian. This is always true for $\coker(\phi)$, and it can be shown to be true for $\ker(\phi)$ using the flatness of $R$. Similarly, given a family $\{M_i:i\in I\}$ of cartesian modules, one can show that the coproduct $\bigoplus_I M_i$ taken in $\mod R$ is cartesian, so that $\bigoplus_I M_i$ is a coproduct in $\cmod R$.\\
As for the category of modules, the most subtle point is to verify that $\cmod R$ has a generator; the rest of the subsection is devoted to the proof of this fact.
\end{proof}

Let us denote by $[1]=\{0\overset{}{\longrightarrow} 1\}$ the category with two objects and just one non-identical morphism between them. A representation $R\colon [1]\to \Add$ is completely determined by the functor $\phi\colon R_0\to R_1$. This allows us to consider morphisms between small preadditive categories just as a representation of $[1]$. Notice also that an $R$-module $M$ is determined by a pair of modules $M_i\in (R_i^{op},\Ab)$ ($i=0,1$) and a morphism $M_\phi\colon M_0\to \phi^*(M_1)$. Furthermore, $M$ is cartesian provided the adjoint morphism $(M_\phi)_!$ is an isomorphism. 

\begin{lemma}\label{tech_coh}
Let $R_0$ and $R_1$ be two small preadditive categories, let $\kappa=\max\{\N,|\mathrm{Mor}(R_0)|, |\mathrm{Mor}(R_1)|\}$, let also $R=\{\phi\colon R_0\to R_1\}$ be a right flat representation of $[1]$ and let $M=\{M_0\to \phi^*(M_1)\}$ be a cartesian module over $R$. Given two families $X_0$ and $X_1$ of elements of $M_0$ and $M_1$ respectively, such that $|X_0|,|X_1|\leq \kappa$, there exists a cartesian submodule $N=\{N_0\to \phi^{*}(N_1)\}\leq M$ such that $|N|\leq \kappa$, $X_0\subseteq N_0$ and $X_1\subseteq N_1$. Furthermore, we can take $N_i$ pure in $M_i$ ($i=0,1$).
\end{lemma}
\begin{proof}
By Corollary \ref{coro_pre_gen} there is $M'\leq M_0$ such that $X_1\leq \phi_!(M')$ and $|M'|\leq \kappa$. We define $N$ as follows: we let $N_0\leq M_0$ be a pure submodule of $M_0$ containing $M'$ and $X_0$ and such that $|N_0|\leq \kappa$, and $N_1=\phi_!(N_0)\leq \phi_!(M_0)=M_1$. The morphism $N_\phi$ is just the restriction of $M_\phi$. It is now easy to see that $N$ satisfies the properties in the statement.
\end{proof}

\begin{proposition}\label{prop.pure}
In the setting of Theorem \ref{cart_groth}, let $\kappa=\sup\{|\N|,|\mor \C|,|\mor R_c|:c\in \Ob \C\}$ and let $M$ be a cartesian right $R$-module. Let $c\in \Ob \C$ and let $x\in M_c$. Then there exists a cartesian submodule $M_x\leq M$ of type $\kappa$ (i.e., $|M_x|\leq \kappa$) such that $x\in(M_x)_c$ and $(M_x)_d$ pure in $M_d$ for all $d\in \Ob\C$.
\end{proposition}
\begin{proof}
Choose a well-ordering for $\mor \C$ and notice that any initial segment in $\mor \C$ has cardinality $\leq \kappa$. We consider also the poset $\N\times \mor \C$ with the lexicographical order. We proceed by induction on $\N\times \mor \C$ to construct a family $\{Y_{n,\alpha} : \alpha\in \mor \C\}$ of $R$-submodules  of $M$ such that:
\begin{enumerate}[\rm (1)]
\item  $x \in (Y_{0,\alpha_0})_c$, where $\alpha_0$ is the least element of $\mor \C$;
\item  if $(n,\alpha)\leq (m,\beta)\in\N\times \mor \C$ then $Y_{n,\alpha}\leq Y_{m,\beta}$;
\item  given $(n,\alpha\colon c\to d)\in\N\times\mor \C$, the module $(Y_{n,\alpha})_c\to \alpha^*(Y_{n,\alpha})_d$ is cartesian (when considered as a module over the representation $R_c\overset{\alpha}{\to} R_d$ of $[1]$) and $(Y_{n,\alpha})_c\leq M_c$ pure for all $c\in\Ob\C$;
\item  $|Y_{n,\alpha}|\leq \kappa$ for all $(n,\alpha)\in\N\times\mor \C$.
\end{enumerate}
Our $M_x$ will be just the direct union (in $\mod R$) $M_x=\bigcup_{(n,\alpha)\in \N\times\mor \C}Y_{n,\alpha}$, one can deduce easily by (1)--(4) that $x\in (M_x)_c$, that $(M_x)_c$ is pure in $M_c$ and that $|M_x|\leq \kappa$. The fact that $M_x$ is cartesian can be shown as follows:  let $(\beta\colon c_1\to c_2)\in\mor \C$ and notice that the countable family $\{Y_{n,\beta}:n\in \N\}$ is cofinal in $\{Y_{n,\alpha}:(n,\alpha)\in\N\times \mor \C\}$, so that there is a canonical isomorphism $M_x\cong \bigcup_{n\in\N}Y_{n,\beta}$. In particular, the map $\beta_!(M_x)_{c_1}\to (M_x)_{c_2}$ is an isomorphism as it can be identified with the direct limit of the isomorphisms $\beta_!(Y_{n,\beta})_{c_1}\to (Y_{n,\beta})_{c_2}$ with $n\in\N$ (here we are using also that $\beta_!$, being a left adjoint, commutes with colimits).

Thus, let us construct inductively the family $\{Y_{n,\alpha}:(n,\alpha)\in\N\times \mor \C\}$. We start constructing $Y_{0,\alpha_0}$, where $\alpha_0\colon c_1\to c_2$ is the least element in $\mor \C$. We proceed by steps:
\begin{enumerate}[\rm (i)]
\item first we let $X^0_{0}(d)=\{M_\alpha(x):\alpha\in \hom_\C(c,d)\}$,
so that $X_0^0(d)$ is a family of elements of $M_d$, for all $d\in \Ob\C$, such that $X_0^0(d)\leq \kappa$;
\item secondly, we use Lemma \ref{tech_coh} to construct $X^0_{1}(c_i)\leq M_{c_i}$ pure so that $X^0_{1}(c_1) \to \alpha_0^*(X^0_{1}(c_2))$ is a cartesian ($R(c_1)\to R(c_2)$)-submodule of $M_{c_1} \to \alpha_0^*M_{c_2}$, such that $X^0_0(c_i)\subseteq X^0_{1}(c_i)$ (for $i=1,2$). Furthermore, if $d\neq c_1,c_2$, we let $X^0_{1}(d)=X^0_{0}(d)$. Notice thus that $X^0_0(d) \subseteq  X^0_{1}(d)$, and $|X^0_{1}(d)|\leq \kappa$ for all $d\in \Ob\C$;
\item now let $T^0_{0,\alpha_0}$ be the minimal $R$-submodule of $M$ such that $X^0_{1}(d)\subseteq (T^0_{0,\alpha_0})_d$, for all $d\in \Ob\C$;
\item for all $n\in\N$, we iterate inductively the above three steps to construct $T^n_{0,\alpha_0}$. Indeed, suppose we already constructed $T^n_{0,\alpha_0}$ for some $n\in\N$ and proceed as follows 
\begin{itemize}
\item let $X^{n+1}_{0}(d)=(T^n_{0,\alpha_0})_d$ for all $d\in\Ob\C$;
\item construct $X^{n+1}_{1}(c_1) \to \alpha_0^*(X^{n+1}_{1}(c_2))$ using Lemma \ref{tech_coh}, and we let $X^{n+1}_{1}(d)=X^{n+1}_{0}(d)$ for $d\neq c_1,c_2$
\item finally let $T^{n+1}_{0,\alpha_0}$ be the minimal $R$-submodule of $M$ such that $X^{n+1}_{1}(d)\subseteq (T^{n+1}_{0,\alpha_0})_d$, for all $d\in \Ob\C$;
\end{itemize}
\item in this way we constructed a chain $\{T^n_{0,\alpha_0}:n\in\N\}$ of $R$-submodules of $M$. We define $Y_{0,\alpha_0}$ to be the union of this ascending chain.
\end{enumerate}
It is clear that $Y_{0,\alpha_0}$ satisfies conditions (1) and (4), while condition (2) is empty in this case. One should only verify property (3), that is, the morphism $(\alpha_0)_!(Y_{0,\alpha_0})_{c_1}\to (Y_{0,\alpha_0})_{c_2}$ is an isomorphism but this follows by construction, in fact there are ascending chains of $R(c_i)$-modules ($i=1,2$)
$$X^{0}_{1}(c_i)\leq (T^{0}_{0,\alpha_0})_{c_i}\leq X^{1}_{1}(c_i)\leq (T^{1}_{0,\alpha_0})_{c_i}\leq \ldots \leq X^{n}_{1}(c_i)\leq (T^{n}_{0,\alpha_0})_{c_i}\leq X^{n+1}_{1}(c_i)\leq \ldots$$
showing that the map $(\alpha_0)_!(Y_{0,\alpha_0})_{c_1}\to (Y_{0,\alpha_0})_{c_2}$
 is in fact the direct union of the isomorphisms $(\alpha_0)_!X^{n}_{1}(c_1) \to X^{n}_{1}(c_2)$ (for $n\in\N$), and it is therefore an isomorphism.
 
\medskip
Now that we have constructed $Y_{0,\alpha_0}$, let us proceed with the inductive step. Indeed, let $(m,\beta\colon c_1\to c_2)\in\N\times\mor \C$ and suppose we have already constructed $Y_{n,\alpha}$ for all $(n,\alpha)< (m,\beta)$. Then we construct $Y_{m,\beta}$ using the same steps as we did for $Y_{0,\alpha}$, with the obvious change of notation from $T^0_{0,\alpha_0}$ to $T^0_{m,\beta}$ in steps (iii) to (v) and changing step (i) by the following
\begin{enumerate}[\rm (i$'$)]
\item first we let $X^0_{0}(d)=\bigcup_{(n,\alpha)< (m,\beta)}(Y_{n,\alpha})_d$,
so that $X_0^0(d)$ is a family of elements of $M_d$, for all $d\in \Ob\C$, such that $X_0^0(d)\leq \kappa$.
\end{enumerate}
In the same way as we did for $Y_{0,\alpha_0}$, one shows that $Y_{m,\beta}$ satisfies properties (1), (3) and (4), while property (2) holds by construction and (i$'$).
\end{proof}

\begin{corollary}
In the setting of Theorem \ref{cart_groth}, $\cmod R$ has a generator.
\end{corollary}
\begin{proof}
Let $\kappa=\sup\{|\N|,|\mor \C|,|\mor R_c|:c\in \Ob \C\}$. By the above proposition, given a cartesian right $R$-module $M$,  $c\in \Ob \C$ and $x\in M_c$, there is a cartesian submodule $M_x\leq M$ of type $\kappa$ such that $x\in(M_x)_c$. Then, it is easily seen that 
$\sum_{c\in \Ob\C}\sum_{x\in M_c}M_x=M$. Thus, any cartesian $R$-module $M$ is the sum of its cartesian $R$-submodules of type $\kappa$. The statement can then be derived taking a set $\F$ of representatives of the cartesian modules of type $\kappa$, so that $\bigoplus_{S\in \F}S$ generates $\cmod R$. 
\end{proof}

\section{The representation theorem}\label{sect.repres}

\subsection{The induced change of base}\label{change_of_base}

Let $R$ and $S$ be two rings, let $\phi\colon R\to S$ be an endomorphism and consider the change of base adjunction:
$$\phi_!\colon \mod R\rightleftarrows\mod S\colon \phi^*\,.$$ 
Using the right exactness of $\phi_!$ it is not difficult to see that $\phi_!(F)$ is finitely presented provided $F$ is finitely presented. In this way we obtain a functor
$$\Phi=(\phi_!)\restriction_{\fp R}\colon \fp R\to \fp S\,.$$
Since $\Phi$ is an additive functor between additive categories, it  induces a  change of base adjunction
$$\Phi_!\colon ((\fp R)^{op},\Ab)\rightleftarrows((\fp S)^{op},\Ab)\colon \Phi^*\,.$$
%Just to be more explicit, remember that $\Phi^*$ acts just pre-composing $\Phi$, and $\Phi_!$ sends a functor $T:(\fp R)^{op}\to\Ab$ to the functor $\Phi_!T:(\fp S)^{op}\to\Ab$ that associates to a finitely presented right $S$-module $G$ the following Abelian group
%$$\Phi_!T(G)=T(-)\otimes_{\fp R}\hom_S(G,\Phi(-))\,.$$
%
%Remember that the contravariant Yoneda functor is defined as follows: $Y:\mod R\to ((\fp R)^{op},\Ab)$, such that $M\mapsto \hom_R(-,M)\restriction_{\fp R}$.
We obtain the following squares:
\begin{equation}\xymatrix{
((\fp R)^{op},\Ab)\ar[r]^{\Phi_!}& ((\fp S)^{op},\Ab) &&((\fp R)^{op},\Ab)& ((\fp S)^{op},\Ab)\ar[l]_{\Phi^*} \\
\mod R\ar[r]^{\phi_!}\ar[u]&\mod S\ar[u] &&\mod R\ar[u]&\mod S\ar[l]_{\phi^*}\ar[u]
}\end{equation}
where the vertical arrows are the contravariant Yoneda embeddings $Y$. 

\begin{proposition}\label{commuting_squares}
In the above setting, given $N\in\mod S$ and $M\in\mod R$,  
\begin{enumerate}[\rm (1)]
\item $\Phi^*Y(N)\cong Y\phi^*(N)$;
\item $\Phi_!Y(M)\cong Y\phi_!(M)$.
\end{enumerate}
Furthermore, the isomorphism in (1) extends to a natural isomorphism of functors $\Phi^*Y\cong Y\phi^*$. 
% re is a natural isomorphism of functors $\omega:\Phi^*Y\tilde\Longrightarrow Y\phi^*$, defined by
%$$\omega_N:\Phi^*Y(N)[=\hom_S(\Phi(-),N)\restriction_{\fp R}]\to Y\phi^*(M)[=\hom_R(-,\phi^*(N))\restriction_{\fp R}]$$
%such that, for all $F\in \fp R$, $(\omega_N)_F(\phi:F\otimes_R S\to N)=(\psi:F\to N\otimes_SS)$, such that $\psi(f)=\phi(f\otimes 1_S)\otimes 1_S$. \NB usare l'hom invece del tensore nel dimenticante???\NB
\end{proposition}
\begin{proof}
(1) Just by definition, 
$$\Phi^*Y(N)=\hom_S(\Phi(-),N)\restriction_{\fp R}=\hom_S(\phi_!(-),N)\restriction_{\fp R}\ \ \text{ and }\ \ \hom_R(-,\phi^*(N))\restriction_{\fp R}=Y\phi^*(N)\,,$$ so the isomorphism $\Phi^*Y(N)\cong Y\phi^*(N)$ follows by the fact that $\phi_!$ is left adjoint to $\phi^*$. It is now clear that this isomorphism is natural and it extends to maps to give the natural isomorphism $\Phi^*Y\cong Y\phi^*$.
%$\cong\,,$$
%where the isomorphism is the $\omega_N$ described in the statement.\NB  All the verifications are standard since $\varepsilon$ is the counit of the usual tensor-hom adjunction. One proves analogously that $\Phi^*Y\cong Y\phi^*$.

\smallskip\noindent
(2) Let now $F$ be a finitely presented right $R$-module and let us prove that $\Phi_!Y(F)=Y\phi_!(F)$. Indeed, for all $G\in\fp S$,
$$\Phi_!Y(F)(G)=\hom_R(-,F)\otimes \hom_S(G,\Phi(-))\cong\hom_S(G,\Phi(F))=\hom_S(G,\phi_!(F))=Y\phi_!(F)(G)\,.$$
Take now a general right $R$-module $M$ and write it as a direct limit of finitely presented modules $M\cong\varinjlim_{i} F_i$. Then, for any finitely presented module $F$ you get $\hom_R(F,\varinjlim_{i} F_i)\cong\varinjlim_{i} \hom_R(F,F_i)$, so there is an isomorphism $\hom_R(-,\varinjlim_{i} F_i)\restriction_{\fp R}\cong \varinjlim_{i} \hom_{\fp R}(-,F_i)$ in $((\fp R)^{op},\Ab)$. Then,
\begin{align*}
\Phi_!Y(M)(G)&=\hom_R(-,M)\restriction_{\fp R}\otimes \hom_S(G,\Phi(-))\\
&\cong(\varinjlim_{i} \hom_{\fp R}(-,F_i))\otimes \hom_S(G,\Phi(-))\\
&\cong\varinjlim_{i} (\hom_{\fp R}(-,F_i)\otimes \hom_S(G,\Phi(-)))\\
&\cong\varinjlim_{i} (\hom_S(G,\Phi(F_i)))\\
&\cong\hom_S(G,\varinjlim_{i}\Phi(F_i))\\
&\cong \hom_S(G,\phi_!(\varinjlim_{i}F_i))= \hom_S(G,\phi_!(M))\,,
\end{align*}
where we used that the tensor product (of modules), the tensor product of functors and the covariant hom-functor corepresented by a finitely presented object commute with direct limits. 
\end{proof}

As a corollary to the above proposition we obtain that both $\Phi_!$ and $\Phi^*$ preserve flat functors:

\begin{corollary}\label{coro_rest_flat}
In the above notation, $\Phi^*$ restricts to a functor 
$$\Flat((\fp S)^{op},\Ab)\to\Flat((\fp R)^{op},\Ab)$$ and $\Phi_!$ restricts to a functor $\Flat((\fp R)^{op},\Ab)\to \Flat((\fp S)^{op},\Ab)$.
\end{corollary} 
\begin{proof}
We have to show that, given an object $\mathcal M$ in $\Flat((\fp S)^{op},\Ab)$, then $\Phi^*(\mathcal M)$ is flat. Indeed, by Crawley-Boevey's Theorem $\mathcal M\cong Y(M)$ for some $M\in \mod S$, and so, by Proposition \ref{change_of_base}, $\Phi^*(\mathcal M)\cong \Phi^*Y(M)\cong Y\phi^*(M)$, which is flat by Crawley-Boevey's Theorem. The second statement follows similarly.
\end{proof}

By the Crawley-Boevey's Theorem, if we restrict the codomain of $Y\colon \mod R\to ((\fp R)^{op},\Ab)$ to the full subcategory of flat functors, then $Y$ has a quasi-inverse $Y^{-1}\colon \Flat((\fp R)^{op},\Ab)\to \mod R$. Using also Corollary \ref{coro_rest_flat}, we obtain the following squares:
 
\begin{equation}\xymatrix@C=20pt{
\Flat((\fp R)^{op},\Ab)\ar[r]^{\Phi_!}\ar@/_10pt/[d]& \Flat((\fp S)^{op},\Ab)\ar@/_10pt/[d] &\Flat((\fp R)^{op},\Ab)\ar@/_10pt/[d]& \Flat((\fp S)^{op},\Ab)\ar[l]_{\Phi^*}\ar@/_10pt/[d] \\
\mod R\ar[r]^{\phi_!}\ar@/_10pt/[u]&\mod S\ar@/_10pt/[u] &\mod R\ar@/_10pt/[u]&\mod S\ar[l]_{\phi^*}\ar@/_10pt/[u]
}\end{equation}

\begin{corollary}
In the above notation, the following natural isomorphisms of functors hold true:
$$(\Phi^*)\restriction_{\Flat((\fp S)^{op},\Ab)}\cong Y\phi^*Y^{-1}\ \ \text{ and }\ \ (\Phi_!)\restriction_{\Flat((\fp R)^{op},\Ab)}\cong Y\phi_!Y^{-1}\,.$$
\end{corollary}
\begin{proof}
For the first isomorphism use the fact that $\Phi^*Y\cong Y\phi^*$ (see Proposition \ref{commuting_squares}) and that, $YY^{-1}\cong \id_{\Flat((\fp S)^{op},\Ab)}$, so that $\Phi^*\cong \Phi^{*}YY^{-1}\cong Y\phi^*Y^{-1}$. For the second isomorphism use the fact that $Y\phi_!Y^{-1}$ is a left adjoint to $Y\phi^*Y^{-1}$ (as the composition of left adjoints is  left adjoint and quasi-inverses are left and right adjoints) and that $(\Phi_!)\restriction_{\Flat((\fp R)^{op},\Ab)}$ is a left adjoint to $(\Phi^*)\restriction_{\Flat((\fp S)^{op},\Ab)}$. Hence, $(\Phi_!)\restriction_{\Flat((\fp R)^{op},\Ab)}\cong Y\phi_!Y^{-1}$, as they are adjoint to isomorphic functors.
\end{proof}

We have now all the background needed to prove the following 

\begin{lemma}\label{cart_coro_Y}
Let $\phi\colon R\to S$ be a ring homomorphism and let $\Phi=(\phi_!)\restriction_{\fp R}\colon \fp R\to \fp S$. Given a morphism $f\colon M\to \phi^* N$, with $M\in\mod R$ and $N\in \mod S$, the following are equivalent:
\begin{enumerate}[\rm (1)]
\item the adjoint morphism $g\colon \phi_!M\to N$ is an isomorphism;
\item the adjoint morphism $G\colon \Phi_!Y(M)\to Y(N)$ to $F=Y(f)\colon Y(M)\to Y(\phi^*N)\cong \Phi^*Y(N)$ is an isomorphism.
\end{enumerate}
%In the setting of Construction \ref{const_[1]}, suppose  we are given a morphism $\phi:M\to \alpha^*N$, with $M\in\mod{R(c)}$ and $N\in\mod {R(d)}$. Let also $\phi'=\alpha^*((\eta_d)_N)\circ \phi:M\to \alpha^*Y_d^{-1}Y_dN$ and 
%$$Y_c\phi':Y_cM\to Y_c\alpha^*Y_d^{-1}Y_dN=A^*Y_dN\,.$$
%Then the adjoint morphism $Y_c(\phi')_!$ to $Y_c(\phi')$ (with respect to $(A_!,A^*)$) is an isomorphism if and only if the adjoint morphism to $\phi$ (with respect to $(\alpha_!,\alpha^*)$) $\phi_!:\alpha_!M\to N$ is an isomorphism.
\end{lemma}
\begin{proof}
Consider the following commutative square:
$$
\xymatrix
{
\hom(YM,\Phi^*YN)\ar[rr]|\cong&&\hom(\Phi_!YM,YN)\\
\hom(YM,Y\phi^*N)\ar[u]|\cong&&\hom(Y\phi_!M,YN)\ar[u]|\cong\\
\hom(M,\phi^*N)\ar[rr]|\cong\ar[u]|\cong&&\hom(\phi_!M,N)\ar[u]|\cong
}
$$
where the horizontal arrows are given by the adjointness of  $(\phi_!,\phi^*)$ and $(\Phi_!,\Phi^*)$, while the vertical arrows are constructed applying the contravariant Yoneda functors and using the natural isomorphisms $\Phi^*Y\cong Y\phi^*$ and $\Phi_!Y\cong Y\phi_!$. By this diagram one can see that the morphism $G$ is $Y(g)$ composed with an isomorphism. Since $Y$ reflects isomorphisms, $G$ is an isomorphism if and only if $g$ is an isomorphism.
\end{proof}

In what follows we will need to work with the fibers of the covariant Yoneda functor  
$$Y^*\colon (\fp R)^{op}\to (\fp R,\Ab)$$ 
that associates to a finitely presented right $R$-module $F$ the covariant hom-functor  $\hom_R(F,-)\restriction_{\fp R}$. In this particular case there is a nice description of the fibers. Indeed, fix an object $H\colon \fp R\to \Ab$ in the functor category, then
\begin{enumerate}[\rm --]
\item the objects of $(\fp R)^{op}_{/H}$ are pairs $(F,f)$, where $F\in\fp R$ and $f\in H(F)$ (in fact, it is possible to identify $\hom_{(\fp R,\Ab)}(\hom_R(F,-),H)$ and $H(F)$);
\item a morphism $T\colon (F,f)\to (F',f')$ in $(\fp R)^{op}_{/H}$ is a homomorphism $T\colon F'\to F$ in $\fp R$ such that $H(T)\colon H(F')\to H(F)$ sends $f'\mapsto f$.
\end{enumerate}

\begin{proposition}\label{induced_exact}
In the above notation, if $\phi_!\colon \mod R\to \mod S$ is exact, then $\Phi_!\colon ((\fp R)^{op},\Ab)\to ((\fp S)^{op},\Ab)$ is exact. 
\end{proposition}
\begin{proof}
For any finitely presented right $S$-module $G$ consider the evaluation functor $\mathrm{ev}_G \colon ((\fp S)^{op},\Ab) \rightarrow \Ab$ such that $\mathrm{ev}_G(F(-))=F(G)$. One can see that $\Phi_!$ is exact if and only if the composition $\ev_G\circ \Phi_!$ is exact for all $G\in \fp S$. On the other hand, there is a natural isomorphism of functors $\ev_G\circ \Phi_!\cong (-\otimes_{\fp R} \Hom_S(G,\Phi(-)))$. Thus, we are reduced to verify that the functor
$$\Hom_S(G,\Phi(-))\colon \fp R\to \Ab$$
is flat (see \cite{flat_func}),which is equivalent to say that the fiber $\mathcal F_G=((\fp R)^{op})_{/\Hom_S(G,\Phi(-))}$ of the Yoneda embedding $(\fp R)^{op}\to (\fp R,\Ab)$ over $\Hom_S(G,\Phi(-))\in (\fp R,\Ab)$ is filtered from above. Indeed,  $\mathcal F_G$ is not empty since it contains the object $(0,0\in\Hom_S(G,0))$. Furthermore, given two objects $(A, a\in\Hom_S(G,\Phi(A))$ and $(A',a'\Hom_S(G,\Phi(A'))$, we can consider  the direct product $A\times A'\in\fp R$ with the two canonical projections $\pi\colon A\times A'\to A$ and $\pi'\colon  A\times A'\to A'$. Then, $\pi$ and $\pi'$ induce morphisms in $\mathcal F_G$ respectively from $(A, a\in\Hom_S(G,\Phi(A))$ and from $(A', a'\in\Hom_S(G,\Phi(A'))$ to $(A\times A', (a,a')\in\Hom_S(G,\Phi(A\times A')))$ (here we are using implicitly the fact that $\Hom_S(G,\Phi(A\times A'))\cong \Hom_S(G,\Phi(A))\times\Hom_S(G,\Phi(A'))$).

It remains to check the condition on a pair of parallel morphisms 
$$\alpha_1,\alpha_2\colon (A, a\in\Hom_S(G,\Phi(A))\rightrightarrows(A', a'\in\Hom_S(G,\Phi(A')))\,,$$
that correspond to two morphisms $\alpha_1,\alpha_2\colon A'\rightrightarrows A$ such that $\Phi(\alpha_1)\circ a'=\Phi(\alpha_2)\circ a'=a$. Indeed, we have to find a finitely presented module $B\in \fp R$, a homomorphism $b\colon G\to \Phi(B)$ and a morphism $\beta\colon B\to A'$ such that $\Phi(\beta)\circ b=a'$ and $\alpha_1\beta=\alpha_2\beta$. Let $\alpha=\alpha_1-\alpha_2$ and consider its kernel $k\colon \ker(\alpha)\to A'$ in $\mod R$ (let us stress the fact that $\ker(\alpha)$ may very well not live in $\fp R$). Since $\phi_!$ is exact, it commutes with kernels and so $\ker(\Phi(\alpha))=\phi_!(\ker(\alpha))$. Thus we obtain a morphism $G\to \phi_!(\ker(\alpha))$ such that the following diagram commutes
$$\xymatrix{
\ker(\Phi(\alpha))\ar[rr]&&\Phi(A')\ar[rr]^{\Phi(\alpha)}&& \Phi(A)\\
G\ar@{.>}[u]^{\exists!}\ar[rru]^{a'}
}$$
Now write $\ker(\alpha)$ as a direct limit of finitely presented modules $\ker(\alpha)=\varinjlim_IB_i$ and notice that, since $G$ is finitely presented, there exists $j\in I$ such that the morphism $G\to \phi_!(\ker(\alpha))=\varinjlim_I\Phi(B_i)$ factors through the structural map $\Phi(B_j)\to \phi_!(\ker(\alpha))$. Let $B=B_j$, denote by $b\colon G\to B$ the map we obtained and let $\beta\colon B\to A'$ be the composition of $B\to \ker(\alpha)$ and $\ker(\alpha)\to A'$. It is now easy to check that $\beta\colon (A',a')\to (B,b)$ satisfies the required conditions. 
\end{proof}

\subsection{The induced representation}

Let $R\colon \C\to \Ring$ be a strict representation. The aim of this section is to define an induced representation $R_{fp}\colon \C\to \Add$ and to construct a functor $Y\colon \cmod R\to \cmod {R_{fp}}$.

\begin{definition}\label{ind_rep_def}
Let $R\colon \C\to \Ring$ be a strict representation. We define the pseudofunctor $R_{fp}\colon \C\to \Add$ as follows:
\begin{enumerate}[\rm --]
\item $R_{fp}(c)=\fp {R_c}$, for all $c\in \Ob\C$;
\item $R_{fp}(\alpha)=\alpha_!\restriction_{\fp {R_c}}\colon \fp {R_c}\to \fp {R_d}$, for all $(\alpha\colon c\to d)\in \C$;
\item given $c\overset{\alpha}{\to} d\overset{\beta}{\to} e$, we let $\mu_{\beta,\alpha}\colon R_{fp}(\beta)R_{fp}(\alpha)\to R_{fp}(\beta\alpha)$ be the morphism described in Lemma \ref{functoriality}, that is, given $F\in \fp {R_c}$, 
\begin{align*}\mu_{\beta,\alpha}\colon R_{fp}(\beta)R_{fp}(\alpha)F=(F\otimes_{R_{c}}R_d)\otimes_{R_d}R_e&\longrightarrow F\otimes_{R_c}R_e=R_{fp}(\beta\alpha)F\\
(f\otimes r_1)\otimes r_2&\longmapsto f\otimes R_\beta(r_1)r_2\,;\end{align*}
\item given $c\in \Ob\C$, $\delta_c\colon \id_{R_{fp}(c)}\to R_{fp}(\id_c)$ is the natural transformation such that, for all $F\in\fp {R_{c}}$,
\begin{align*}
(\delta_c)_F\colon \id_{R_{fp}(c)}F=F&\longrightarrow F\otimes_{R_c}R_c=R_{fp}(\id_c)\\
f&\longmapsto f\otimes 1_{R_c}\,.
\end{align*}
\end{enumerate}
\end{definition}

It is an easy exercise on tensor products in categories of modules to check the axioms (Rep.1) and (Rep.2) that make $R_{fp}$ a representation. 

\begin{lemma}\label{always_left_flat}
In the notation of Definition \ref{ind_rep_def}, $R$ is a right flat representation if and only if $R_{fp}$ is right flat.  Furthermore, $R_{fp}$ is always left flat.
\end{lemma}
\begin{proof}
The first part of the statement is a direct consequence of Proposition \ref{induced_exact}, so it remains to show that $R_{fp}$ is always left flat. Indeed, given $(\alpha\colon c\to d)\in \C$ we should prove that the functor 
$$\hom_{\fp{R_d}}(R_{fp}(\alpha)(-),-)\otimes_{\fp {R_c}}-\colon (\fp {R_c},\Ab)\to (\fp {R_d},\Ab)$$
is exact. Equivalently, we have to prove that 
$$\hom_{\fp{R_d}}(R_{fp}(\alpha)(-),K)\colon (\fp {R_c})^{op}\to \Ab$$ 
is a flat functor for all $K\in\fp {R_d}$. Since $(\fp {R_c})^{op}$ is finitely complete, it is enough to show that 
$$\hom_{\fp{R_d}}(R_{fp}(\alpha)(-),K)\colon (\fp {R_c})^{op}\to \Ab$$ 
is left exact (see \cite{flat_func}), which is true as $\hom_{\fp{R_d}}(R_{fp}(\alpha)(-),K)$ is the composition of the left exact\footnote{Indeed, $\alpha_!$ is right exact as it is left adjoint, so that $R_{fp}(\alpha)$ is right exact, as it is the restriction of $\alpha_!$ to a finitely cocomplete subcategory, thus, $R_{fp}(\alpha)^{op}$ is left exact.} functor $R_{fp}(\alpha)^{op}\colon (\fp {R_c})^{op}\to (\fp {R_d})^{op}$ followed by  $\hom_{\fp{R_d}}(-,K)\colon (\fp {R_d})^{op}\to \Ab$.
\end{proof}

We can now start to relate the two categories $\mod R$ and $\mod {R_{fp}}$. 

\begin{definition}\label{yoneda_global}
Let $R\colon \C\to \Ring$ be a strict representation and let $Y_c\colon \mod R_c\to ((\fp R_c)^{op},\Ab)$ be the contravariant Yoneda functor, for all $c\in \Ob\C$. We define a functor $Y\colon \mod R\to \mod{R_{fp}}$ as follows:
\begin{enumerate}[\rm --]
\item given a right $R$-module $M$, we let $Y(M)$ be such that $Y(M)_c=Y_c(M_c)$, for all $c\in \Ob\C$. Furthermore, 
$$Y(M)_\alpha\colon Y_c(M_c)\overset{Y_c(M_\alpha)}{\longrightarrow}  Y_c\alpha^*M_d\overset{adjunction}{\longrightarrow} A^*Y_dM_d\,,$$ 
for all $\alpha\in \C$, where the map $Y_c\alpha^*M_d\to A^*Y_cM_c$ sends a morphism $f\in\hom_{R_c}(N,\alpha^*M_d)$ to its adjoint morphism $(\varepsilon_\alpha)_{M_c}\circ\alpha_!(f)\in\Hom_{R_d}(\alpha_!N,M_d)=A^*Y_dM_d$ (here $\varepsilon_\alpha$ is the counit of the adjunction $(\alpha_!,\alpha^*)$ as described at the beginning of Section \ref{change_of_base_sec});
\item given a morphism $\phi\colon M\to N$ in $\mod R$, we let $Y(\phi)\colon Y(M)\to Y(N)$ be the morphism such that $Y(\phi)_c=Y_c(\phi_c)$, for all $c\in \Ob\C$.
\end{enumerate}
\end{definition}

Let us verify that the above definition is correct:

\begin{lemma}
In the notation of Definition \ref{yoneda_global}, $Y\colon \mod R\to \mod{R_{fp}}$ is a well-defined functor.
\end{lemma}
\begin{proof}
Given a right $R$-module $M$, let us verify that $Y(M)$ is a right $R_{fp}$-module. Let us fix some notation first: given $(\alpha\colon c\to d),\, (\beta\colon d\to e)\in\C$, let $\gamma=\beta\alpha\colon c\to e$ be their composition, let $(\alpha_!,\alpha^*)$, $(\beta_!,\beta^*)$ and $(\gamma_!,\gamma^*)$ be the change of base adjunctions relative respectively to $R_\alpha\colon R_c\to R_d$, $R_\beta\colon R_d\to R_e$ and $R_\gamma\colon R_c\to R_e$, while we let $(A_!,A^*)$, $(B_!,B^*)$ and $(C_!,C^*)$ be the change of base adjunctions relative respectively to $R_{fp}(\alpha)\colon \fp {R_c}\to \fp {R_d}$, $R_{fp}(\beta)\colon \fp {R_d}\to \fp{R_e}$ and $R_{fp}(\gamma)\colon \fp {R_c}\to \fp{R_e}$.

\smallskip
Let us verify first the axiom (Mod.1), that is, we should show that the following diagram commutes:
$$
\xymatrix{
Y_cM_c\ar@/_32pt/[rrrrdd]|{Y_cM_\gamma}\ar[rr]^{Y(M)_\alpha}\ar[rrd]|{Y_cM_\alpha}&&A^*Y_dM_d\ar[rr]^{A^*Y(M)_\beta}\ar[rrd]|{A^*Y_dM_\beta}&&A^*B^*Y_eM_e\\
&&Y_c\alpha^*M_d\ar[u]|{\varepsilon_\alpha\circ\alpha_!(-)}\ar[rrd]|{Y_c\alpha^*M_\beta}&&A^*Y_d\beta^*M_e\ar[u]|{\varepsilon_\beta\circ\beta_!(-)}&&C^*Y_eM_e\ar@/_10pt/[ull]|{\mu_{\beta,\alpha}*\id_{M_e}}\\
&&&&Y_e\alpha^*\beta^*M_e(=Y_e\gamma^*M_e)\ar[u]|{\varepsilon_\alpha\circ\alpha_!(-)}\ar@/_10pt/[rru]|{\varepsilon_\gamma\circ\gamma_!(-)}&&}
$$
In fact, almost everything commutes either by definition, naturality of transformations, or since $R$ is a strict representation; let us only comment on the commutativity of the triangle on the right-hand side. Indeed, by Lemma \ref{unit_of_two}, $\varepsilon_\beta\circ \beta_!(\varepsilon_\alpha)_{\beta^*}=\varepsilon_\gamma\circ(\tau_{\beta,\alpha})_{\gamma^*}\circ \beta_!\alpha_!(\sigma_{\beta,\alpha})=\varepsilon_\gamma\circ (\tau_{\beta,\alpha})_{\gamma^*}$ (in fact, $\sigma_{\beta,\alpha}$ is the identity since $R$ is strict). Furthermore, in our case $\tau_{\beta,\alpha}$ (relative to the representation $R$) and $\sigma_{\beta,\alpha}$ (relative to the representation $R_{fp}$) coincide. Now it is easy to check commutativity on elements.

\smallskip
Verifying the axiom (Mod.2) is just an exercise on definitions and tensor products. Similarly, it is not hard to check that $Y$ respects composition of morphisms and identities. Thus, $Y$ is a well-defined functor as desired.
\end{proof}

The following corollary, which is a direct consequence of Lemma \ref{cart_coro_Y}, shows that the functor $Y\colon \mod R\to \mod {R_{fp}}$ described above, induces by restriction a functor $\cmod R\to \cmod{R_{fp}}$.

\begin{corollary}\label{coro_cart}
In the notation of Definition \ref{yoneda_global}, let $M\in\mod R$. Then, $M$ is cartesian if and only if $Y(M)$ is a cartesian.
\end{corollary}

\subsection{The Representation Theorem}

We can now state our main result after introducing the following notation: 

\begin{definition}\label{def_flat_cart_rep}
Let $R\colon \C\to \Ring$ be a strict representation of the small category $\C$ and let $R_{fp}\colon \C\to \Add$ be the induced representation (see Definition \ref{ind_rep_def}). Let $\lFlat(\mod {R_{fp}})$ (resp., $\lFlat(\cmod {R_{fp}})$) be the full subcategory of $\mod R_{fp}$ ($\cmod {R_{fp}}$), whose objects are the (cartesian) modules $M$ such that $M_c\in \Flat((\fp {R_c})^{op},\Ab)$, for all $c\in \Ob\C$. 
\end{definition}

The following corollary is a direct consequence of Proposition \ref{flat_obj_loc}:

\begin{corollary}
In the notation of Definition \ref{def_flat_cart_rep}, let $M\in \mod {R_{fp}}$. If $\C$ is a poset, then $M\in \lFlat(\mod {R_{fp}})$ if and only if $M$ is flat in the sense of Definition \ref{def_flat_module_rep}.
\end{corollary}
\begin{proof}
By Proposition \ref{always_left_flat}, $R_{fp}$ is a left flat representation of a poset, so that Proposition \ref{flat_obj_loc} directly applies to give the desired conclusion.
\end{proof}

\begin{theorem}\label{locally_CB_rep}
Let $\C$ be a small category and let $R\colon \C\to \Ring$ be a strict representation. Then the Yoneda functor $Y\colon \mod R\to \mod R_{fp}$ induces equivalences
$$\mod R\cong \lFlat(\mod {R_{fp}})\ \  \text{ and }\ \ \cmod R\cong \lFlat(\cmod {R_{fp}})\,.$$
\end{theorem}
\begin{proof}
Using Theorem \ref{CB_rep_Th}, it is easy to see that the image of $Y$ is in $\lFlat(\mod {R_{fp}})$. Let us show first that $Y$ induces the  equivalence on the left-hand side. Indeed, we have to show that $Y$ is fully faithful and that the image of $Y$ is isomorphism-dense in  $\lFlat(\mod {R_{fp}})$:

\smallskip
{\em Essential surjectivity}: given $\mathcal M\in\lFlat(\mod {R_{fp}})$, $\mathcal M_c\in\Flat((\fp {R_c})^{op},\Ab)$ for all $c\in \Ob \C$ and so we can define $M_c=Y_c^{-1}(\mathcal M_c)=\mathcal M_c(R_c)$ (so that $\mathcal M_c\cong \hom_{R(c)}(-,M_c)\restriction_{\fp {R(c)}}$). 

Let us briefly describe the right $R_c$-module $Y_c^{-1}(A^*\mathcal M_d)=(A^*\mathcal M_d)(R_c)$. Indeed, as an Abelian group this is exactly $\mathcal M_d(R_c\otimes_{R_c}R_d)$. Furthermore, given $r\in R_c$, consider the following homomorphism of right $R_d$-modules:
$$\lambda_r\colon R_c\otimes_{R_c}R_d\to R_c\otimes_{R_c}R_d \ \text{ such that } \ \lambda_r(r_1\otimes r_2)=(rr_1)\otimes r_2\,.$$
Since $\mathcal M_d$ is contravariant, we obtain a right $R_c$-module structure on  $\mathcal M_d(R_c\otimes_{R_c}R_d)$, where the right multiplication by $r\in R_c$ acts as $\mathcal M_d(\lambda_r)$. Similarly, $\alpha^*\mathcal M_d(R_d)$ is a right $R_c$-module, where the right multiplication by $r\in R_c$ acts as $\mathcal M_d(\lambda_{R_\alpha(r)})$. 

Applying $\mathcal M_d$ to the isomorphism of left $R_c$-right $R_d$-modules $R_c\otimes_{R_c}R_d\to R_d$ such that $(r_1\otimes r_2)\mapsto R_\alpha(r_1)r_2$, with obtain an isomorphism of right $R_c$-modules 
$$\xi_\alpha\colon Y_c^{-1}(A^*\mathcal M_d)\to \alpha^*Y_d^{-1}\mathcal M_d\,.$$

We define $M_\alpha\colon M_c\to \alpha^*M_d$ as the following composition:
$$\xymatrix{ M_c=Y_c^{-1}\mathcal M_c\ar[rr]^{Y_c^{-1}\mathcal M_\alpha}&&Y_c^{-1}(A^*\mathcal M_d)\ar[rr]^{\xi_\alpha}&&\alpha^*Y_d^{-1}\mathcal M_d=\alpha^*M_d}\,.$$

With this definition, it is not difficult to show that $M$ is a right $R$-module and that $Y(M)$ is isomorphic to $\mathcal M$ since, being $\mathcal M$ locally flat, $\mathcal M_c\cong Y_cY_c^{-1}\mathcal M_c$, for all $c\in\Ob\C$. Thus, any locally flat right $R_{fp}$-module is isomorphic to a module in the image of $Y$. 

\smallskip
{\em Full faithfulness}: let $\Phi\colon \mathcal M\to \mathcal N$ be a morphism in $\lFlat(\mod {R_{fp}})$. Then, we can define $\phi\colon M\to N$ (with $M_c=Y_c^{-1}\mathcal M_c$ and $N_c=Y_c^{-1}\mathcal M_c$ as above), as the morphism such that $\phi_c=Y_c^{-1}\Phi_c$, then $Y(\phi)=\Phi$, proving that $Y$ is full. For the faithfulness it is enough to use that each $Y_c$ is faithful. 

\smallskip
The second equivalence in the statement  follows using the already proved full faithfulness, and the fact that $Y(M)\in \cmod {R_{fp}}$ if and only if $M\in\cmod R$ (by Corollary \ref{coro_cart}), so that the essential surjectivity proved above restricts to cartesian modules. 
\end{proof}

The above theorem allows for a good purity theory in $\mod R$ and $\cmod R$, as the following definition and corollary show:

\begin{definition}\label{def.pure}
Let $\C$ be a small category and let $R\colon \C\to \Ring$ be a strict representation.  A short exact sequence $0\to N\to M\to M/N\to 0$ in $\mod R$ (resp., $\cmod R$) is said to be {\em pure-exact} provided $0\to N_c\to M_c\to M_c/N_c\to 0$ is pure-exact for all $c\in\Ob\C$.
\end{definition}

In fact, when $\C$ is a poset, the above notion  coincides with the intuitive idea of pure exactness:

\begin{lemma}
Let $\C$ be a poset and $R\colon \C\to \Ring$ be a strict representation. Then a short exact sequence $0\to N\to M\to M/N\to 0$ in $\mod R$ (resp., $\cmod R$) is pure exact if and only if 
$$0\to N\otimes X\to M\otimes X\to M/N\otimes X\to 0$$
is a short exact sequence of Abelian groups for all $X\in \mod R$. 
\end{lemma}

The following corollary is a direct consequence of Theorem \ref{locally_CB_rep} and Corollary \ref{purity_th}:

\begin{corollary}\label{cor.pure}
Let $R\colon \C\to \Ring$ be a strict representation of the small category $\C$. Then a short exact sequence $0\to N\to M\to M/N\to 0$ in $\mod R$ is  pure-exact if and only if $0\to Y(N)\to Y(M)\to Y(M/N)\to 0$ is exact in $\mod {R_{fp}}$. 
\end{corollary}

\section{Pure injective envelopes and the pure derived category}
\subsection{Covers, envelopes and cotorsion pairs}
Throughout this section the symbol $\mathcal G$ will denote an exact category.
\begin{definition}
Let $\mathcal{L}$ be a strictly full subcategory of $\mathcal{G}$. A morphism
$\phi \colon  M \rightarrow L$ of $\mathcal{G}$ is said to be an
\emph{$\mathcal{L}$-preenvelope} of $M$ if $L \in \mathcal{L}$ and if
$\Hom(L,L')\rightarrow \Hom (M,L') \rightarrow 0$ is exact for every
$L' \in \mathcal{L}$. If any morphism $f\colon L \rightarrow L$ such that
$f\circ \phi= \phi$ is an isomorphism, then it is called an
\emph{$\mathcal{L}$-envelope} of $M$. If the class $\mathcal{L}$ is such that
every object has an $\mathcal{L}$-(pre)envelope, then $\mathcal{L}$ is
called a \emph{(pre)enveloping class}. The dual notions are those of
$\mathcal{L}$-(pre)covers and (pre)covering class.
\end{definition}
The (pre)envelopes and (pre)covers use to take the name of the class over which they are constructed. Thus the notions of injective (pre)envelopes, pure-injective (pre)envelopes, flat (pre)covers, etc. appear naturally in the exact categories where the corresponding classes can be defined.

Given a class of objects $\mathcal F$ in $\mathcal G$, we will denote by $\mathcal F^\perp$ the class
$$\mathcal F^\perp=\{C\in \Ob\mathcal G: \ \Ext^1_{\mathcal G}(F,C)=0,\ \forall F\in \mathcal F\},$$ and by $^{\perp}\mathcal F$ the class 
$$^{\perp}\mathcal F=\{G\in \Ob\mathcal G:\ \Ext^1_{\mathcal G}(G,D)=0,\ \forall D\in \mathcal F\}.$$
A pair of classes $(\mathcal F,\mathcal C)$ in $\mathcal G$ is called a \emph{cotorsion pair} provided that $\mathcal F^{\perp}=\mathcal C$ and $^{\perp}\mathcal C=\mathcal F$.  The cotorsion pair $(\mathcal F,\mathcal C)$ is said to be \emph{complete} provided that for each $M\in \mathcal G$ there are admissible short exact sequences
$0\to C\to F\to M\to 0$ and $0\to M\to C'\to F'\to 0$ where $F,F'\in \mathcal F$ and $C,C'\in \mathcal C$. Then $0\to M\to C'$ is a $\mathcal C$-preenvelope of $M$ with cokernel in $\mathcal F$. Such preenvelopes are named \emph{special} preenvelopes. Dually, $F\to M\to 0$ is a special $\mathcal F$-precover.

A class $\mathcal F$ in $\mathcal G$ is said to be \emph{resolving} if it is closed under kernels of admissible epimorphisms in $\mathcal F$.
A cotorsion pair $(\mathcal F,\mathcal C)$ in $\mathcal G$ is called \emph{hereditary} whenever $\Ext^n_{\mathcal G}(F,C)=0$, for all $F\in \mathcal F$, $C\in \mathcal C$ and $n\geq 1$.

A pair of classes $(\mathcal F,\mathcal F^{\perp})$ is \emph{cogenerated by a set $\mathcal S\subseteq \mathcal F$} provided that $C\in \mathcal F^\perp $ if and only if $ \Ext^1_{\mathcal G}(F,C)=0,\ \forall F\in \mathcal S$. If the pair $(\mathcal F,\mathcal F^{\perp})$ is cogenerated by a set, and $\mathcal F$ is closed under extensions and direct limits (indeed a weaker condition suffices) we get from \cite[Theorem 2.5]{EEGO2} that every object $M$ has a special $\mathcal F^{\perp}$-preenvelope. One standard way of getting the ``cogenerated by a set" condition, is to check that the class $\mathcal F$ is \emph{Kaplansky} and closed under direct limits. We recall now the definition of a Kaplansky class:
\begin{definition}
Let $\mathcal G$ be a Grothendieck category. Let $\mathcal F$ be a class of objects in $\mathcal G$ and let $\kappa$ be a regular cardinal. We say that $\mathcal F$ is a \emph{$\kappa$-Kaplansky class} if for each $Z\subseteq F$, with $F\in \mathcal F$ and where $Z$ is $\kappa$-presentable, there exists a $\kappa$-presentable object $S\neq 0$ such that $Z\subseteq S\subseteq F$ and $S,F/S\in \mathcal F$. We say that $\mathcal F$ is a {\em Kaplansky class} if it is $\kappa$-Kaplanksy for some regular cardinal $\kappa$.
\end{definition}
\subsection{Pure injective envelopes}
The first application of Theorem \ref{locally_CB_rep} is the existence of pure-injective envelopes in $\cmod R$. We recall that $E\in \cmod R$ is \emph{pure-injective} if every pure exact sequence  $0\to E\to M\to N\to 0$ in $\cmod R$ (in the sense of Definition \ref{def.pure}) splits.

The cartesian modules belonging to $\lFlat(\cmod {R_{fp}})^{\perp}$ will be called \emph{cotorsion} modules.
\begin{lemma}\label{lem.Kap}
Let $R\colon \C\to \Ring$ be a strict representation. The class $ \lFlat(\cmod {R_{fp}})$ in $\cmod {R_{fp}}$ is a Kaplansky class.
As a consequence every $M\in \cmod {R_{fp}}$ has a cotorsion envelope whose cokernel belongs to  $\lFlat(\cmod {R_{fp}})$.
\end{lemma}
\begin{proof}
The first part follows as in Proposition \ref{prop.pure}, using Appendix A. Then, since $\lFlat(\cmod {R_{fp}})$ is closed under extensions and direct limits, a standard argument (see \cite[Proposition 2]{flat_cover_conj}) shows that the pair $(\lFlat(\cmod {R_{fp}}, \lFlat(\cmod {R_{fp}})^{\perp})$ is cogenerated by a set. Therefore we can apply \cite[Theorem 2.5]{EEGO2} to infer that every $M\in \cmod {R_{fp}}$ has a cotorsion preenvelope with cokernel in $\lFlat(\cmod {R_{fp}})$. Now, by \cite[Proposition 2.2.1 and Theorem 2.2.2]{Xu} it follows that $M$ has a cotorsion envelope with cokernel in $\lFlat(\cmod {R_{fp}}$.
\end{proof}
\begin{theorem}\label{theor.purei.env}
Let $R\colon \C\to \Ring$ be a strict representation. Every $M\in \cmod {R_{fp}}$ has a pure-injective envelope.
\end{theorem}
\begin{proof}
Once we have established the equivalence $\cmod R\cong \lFlat(\cmod {R_{fp}})$ the proof follows the same lines as the proof given by Herzog in \cite[Theorem 6]{Herzog}; the main ingredients being Corollary \ref{cor.pure} and the fact that pure-injectives in $\cmod R$ are in 1-1 correspondence with cotorsion cartesian modules in $\lFlat(\cmod {R_{fp}})$. Then, given $M\in \cmod R$, its pure injective envelope is the cartesian module $E\in\cmod R$, such that $Y(E)$ is the cotorsion envelope of $Y(M)$ in $\lFlat(\cmod {R_{fp}})$.
\end{proof}

\subsection{Induced cotorsion pairs in a Grothendieck category}
\begin{proposition}\label{prop.induc}
Let $\mathcal G$ be a Grothendieck category. Let $\mathcal F$ be a strictly full additive subcategory which is resolving and closed under taking extensions. Let us denote by $(\mathcal F,\Ext|_{\mathcal F})$ the induced exact category on $\mathcal F$.
Consider a subclass $\mathcal L\subseteq \mathcal F$ such that:
\begin{enumerate}[\rm (1)]
\item $\mathcal L$ is Kaplansky;
\item $\mathcal L$ is closed under direct limits and extensions;
\item $\mathcal L$ contains a generator of $(\mathcal F,\Ext|_{\mathcal F})$.
\end{enumerate}
Then the induced pair $(\mathcal L,\mathcal L^{\perp}\cap \mathcal F)$ is a hereditary and complete cotorsion pair in $(\mathcal F,\Ext|_{\mathcal F})$.
\end{proposition}
\begin{proof}
Let us denote by $\perp_{\mathcal F}$ the orthogonal computed inside $(\mathcal F,\Ext|_{\mathcal F})$. Therefore
$$\mathcal L^{\perp_{\mathcal F}}=\{C\in \mathcal F:\ \Ext^1_{\mathcal G}(L,C)=0, \forall L\in \mathcal L\}.$$ 
It is then clear that $\mathcal L^{\perp_{\mathcal F}}=\mathcal L^{\perp}\cap\mathcal F$. Now let $S\in ^{\perp_{\mathcal F}}\!\!\!\mathcal L$, so $S\in \mathcal F$. We need to show that $S\in \mathcal L$. By the hypothesis, there exists an epimorphism $L\to S\to 0$, with $L\in \mathcal L$. Since $\mathcal F$ is resolving, $K=\ker(L\to S)\in \mathcal F$.  Now, there exists a short exact sequence $0\to K\to D\to M\to 0$ with $M\in \mathcal L$ and $D\in \mathcal L^{\perp}$. Let us construct the pushout diagram along $K\to D$ and $K\to L$:  $$\xymatrix{ & 0\ar[d] & 0\ar[d] &  &\\
0\ar[r] & K\ar[d]\ar[r] & L \ar[r] \ar[d]\ar[r] &S\ar@{=}[d]\ar[r] &0\\
0\ar[r] & D\ar[d]\ar[r] & T\ar[d] \ar[r] &S\ar[r] &0 \\ & M\ar@{=}[r]\ar[d] & M\ar[d] &  &\\& 0 & 0 &  & }$$
Then, since $L,M\in \mathcal L$, it follows that $T\in\mathcal L$. Analogously, since $K,M\in \mathcal F$, $D\in \mathcal L^{\perp}\cap\mathcal F=\mathcal L^{\perp_{\mathcal F}}$. So the sequence $0\to D\to T\to S\to 0$ splits and hence $S\in \mathcal L$.

Now, let us see that $(\mathcal L,\mathcal L^{\perp_{\mathcal F}})$ is a complete cotorsion pair in $(\mathcal F,\Ext|_{\mathcal F})$. Let $F\in \mathcal F$. By the hypothesis on $\mathcal L$ there exists a short exact sequence $0\to F\to T\to L\to 0$ with $T\in \mathcal L^{\perp}$ and $L\in \mathcal L$. Since $\mathcal L\subseteq \mathcal F$ and $\mathcal F$ is closed under extensions, we follow that $T\in \mathcal L^{\perp_{\mathcal F}}=\mathcal L^{\perp}\cap\mathcal F$. So the cotorsion pair $(\mathcal L,\mathcal L^{\perp}\cap \mathcal F)$ in $(\mathcal F,\Ext|_{\mathcal F})$ has enough injectives. Now let $S\in \mathcal F$. The same proof as before gives us a short exact sequence $0\to D\to T\to S\to 0$ with $T\in \mathcal L$ and $D\in\mathcal L^{\perp_{\mathcal F}}$. Therefore $(\mathcal L,\mathcal L^{\perp_{\mathcal F}})$ has enough projectives, whence it is a complete cotorsion pair in $(\mathcal F,\Ext|_{\mathcal F})$.

Let us finally see that the complete cotorsion pair $(\mathcal L,\mathcal L^{\perp_{\mathcal F}})$ is also hereditary. Let us see that $\Ext^2_{\mathcal G}(L,C)=0$, for each $ L\in \mathcal L$ and $C\in \mathcal L^{\perp_{\mathcal F}}$. Then the argument will follow by an easy induction. So let $0\to C\to T\to H\to L\to 0$ be an extension in $\Ext^2_{\mathcal G}(L,C)$. Since $\mathcal L$ contains a generator of $\mathcal F$ and $\mathcal F$ is resolving, we can find an equivalent sequence $0\to C\to T'\to H'\to L\to 0$ to the given one, with $H'$ (and hence $\ker(H'\to L)$) in $\mathcal L$. So it is equivalent to the zero element in $ \Ext^2_{\mathcal G}(L,C)$.
\end{proof}

Now we can apply Proposition \ref{prop.induc} to some particular instances to get induced cotorsion pairs. Since we will be always dealing with the class $\lFlat(\cmod {R_{fp}})$ in $\cmod {R_{fp}}$, we will denote it simply by $\lFlat$. We introduce the notation $\Ch(\lFlat)$ for the chain complexes $M\in \Ch(\cmod {R_{fp}})$ such that $M_n\in \lFlat$, $\forall n\in \mathbb{Z}$. And we denote by $\widetilde{\lFlat}$ the class of acyclic complexes $L\in \Ch(\lFlat)$ for which $Z_n L\in \lFlat$. Here $Z_nL$  stands for the \emph{n-th cycle cartesian module} of $L$. 

\begin{corollary}
Let $R$ be a strict representation and $\cmod {R_{fp}}$ the induced category of cartesian modules on $R_{fp}$:
\begin{enumerate}[\rm (1)]
\item  the pair $(\lFlat,\lFlat^{\perp}\cap \lFlat)$ in $(\lFlat,\Ext|_{\lFlat})$ is a hereditary complete cotorsion pair;
\item  the pair $(\Ch(\lFlat),\Ch(\lFlat)^{\perp}\cap \Ch(\lFlat))$ in $(\Ch(\lFlat),\Ext|_{\Ch(\lFlat)})$ is a hereditary complete cotorsion pair;
\item  the pair $(\widetilde{\lFlat},\widetilde{\lFlat}^{\perp}\cap  \Ch(\lFlat))$ in $(\Ch(\lFlat),\Ext|_{\Ch(\lFlat)})$ is a hereditary complete cotorsion pair.
\end{enumerate}

\end{corollary}
\begin{proof}
We will exhibit, in each of the three cases, two classes $\mathcal L$ and $\mathcal F$ that satisfy the hypotheses of Proposition \ref{prop.induc}, which gives therefore the statement.

(1) Take $\mathcal L=\mathcal F=\lFlat$ and $\mathcal G=\cmod {R_{fp}}$ in Proposition \ref{prop.induc}. By Appendix A, the class $\lFlat$ is clearly closed under kernel of epimorphisms in $\cmod {R_{fp}}$. Clearly, it is also closed under extensions and direct limits. Finally, $\mathcal L$ is Kaplansky by Lemma \ref{lem.Kap}.

(2) Take $\mathcal G=\Ch(\cmod {R_{fp}})$, $\mathcal L=\mathcal F=\Ch(\lFlat)$ in Proposition \ref{prop.induc}. Since $\cmod {R_{fp}}$ is Grothendieck (Theorem \ref{cart_groth}), $\Ch(\cmod {R_{fp}})$ is also Grothendieck. Since $\lFlat$ is resolving, closed under extensions and direct limits, the class $\mathcal F$ will also fulfill these properties. Finally, the class $\mathcal L$ is Kaplansky by \cite[Corollary 2.7 and Theorem 4.2(1)]{Sto}.

(3) Take $\mathcal G=\Ch(\cmod {R_{fp}})$, $\mathcal L=\widetilde{\lFlat}$ and $\mathcal F=\Ch(\lFlat)$ in Proposition \ref{prop.induc}. The class $\mathcal L$ is closed under extensions and direct limits because the acyclic chain complexes form a thick subcategory in $\Ch(\cmod {R_{fp}})$ and $\lFlat$ is closed under extensions and direct limits. The chain complexes of $\mathcal L$ of the form $\ldots \to 0\to G\stackrel{1_G}{\to}G\to 0\to \ldots$, with $G\in \lFlat$, form a generating set for $(\mathcal F,\Ext|_{\Ch(\lFlat)})$, so condition (3) in Proposition \ref{prop.induc} is also satisfied. Finally $\mathcal L$ is Kaplansky by \cite[Corollary 2.7 and Theorem 4.2(2)]{Sto}.
\end{proof}

\subsection{The pure derived category of cartesian modules}

In this subsection we apply Theorem \ref{locally_CB_rep} to define the derived category of $\cmod R$ relative to the pure-exact structure we introduced in Definition \ref{def.pure}. As we already commented in the Introduction, this notion extends and relates the various proposals of pure derived categories on a scheme already appeared in the literature (see \cite[\S 1. Corollary]{EGO} for a quasi-separated scheme and \cite[\S\S 2.5]{MS} for a Noetherian separated scheme).

\begin{theorem}\label{theor.der.pur}
Let $R\colon \C\to \Ring$ be a strict representation. Let $\mathcal E$ be the pure-exact structure  in $\cmod R$ coming from  Definition \ref{def.pure}, and let $\Ch(\cmod R)$ be the category of unbounded complexes of cartesian modules. Then there is an exact injective model category structure on $\Ch(\cmod R)$ with respect to the induced degree-wise exact structure from $\mathcal E$, such that
\begin{enumerate}[\rm --]
\item the trivial objects are the pure acyclic complexes;
\item every chain complex is cofibrant;
\item  the trivially fibrant objects are the injective complexes in $\Ch(\cmod R),\mathcal E)$, that is, the contractible complexes with pure-injective components.
\end{enumerate}
The corresponding homotopy category is the pure derived category $\mathcal D_{\rm pure}(\cmod R)$.
\end{theorem}

\begin{proof}
Let us denote $\perp_{\Ch(\lFlat)}$ by $\perp_{\Ch(\mathcal F)}$. First we will show that the pairs $(\Ch(\lFlat),\Ch(\lFlat)^{\perp_{\Ch(\mathcal F)}})$ and 
$(\widetilde{\lFlat},\widetilde{\lFlat}^{\perp_{\Ch(\mathcal F)}})$ in $\Ch(\lFlat)$ satisfy the Hovey correspondence in its version for exact categories (see \cite{hovey2} and \cite[Corollary 3.4]{G4}). We will denote by $\mathcal R$ the class $\widetilde{\lFlat}^{\perp_{\Ch(\mathcal F)}}$ and by $\mathcal W$ the class $\widetilde{\lFlat}$. It is clear that $\mathcal W$ is a thick subcategory of $\Ch(\lFlat)$. That is, if two terms in a short exact sequence $0\to W_1\to W_2\to W_3\to $ in $\Ch(\lFlat)$ are in $\mathcal W$, then the third one is also in $\mathcal W$. Notice that the chain complexes in $\Ch(\lFlat)^{\perp_{\Ch(\mathcal F)}}$ are the contractible complexes in $\Ch(\lFlat)$ that are cotorsion in each degree. We ought to see that this class coincides with $\mathcal R\cap\mathcal W$. By its definition, for each $G\in \mathcal R\cap \mathcal W$, the identity map $1_G\colon G\to G$ is null homotopic. Hence the chain complex $G$ is contractible. Finally let us see that each degree $G_n$ in $G$ is cotorsion. So let $F\in \lFlat$. Then $\Ext_{\cmod {R_{fp}})}(F,G_n)\cong\Ext_{\Ch(\cmod {R_{fp}})}(D^n(F),G)$, where $D^n(F)$ is the chain complex with $F$ in degree $n$ and $n-1$ and 0 otherwise. All the maps are zero except $d_n=1_F$. Therefore $D^n(F)\in \widetilde{\lFlat}$, so $\Ext_{\Ch(\cmod {R_{fp}})}(D^n(F),G)=0$. Thus $G_n\in \lFlat$ is cotorsion. So these pairs define a model structure in $(\Ch(\lFlat),\Ext_{\Ch(\lFlat)})$. The model structure is injective in the sense that all complexes of $\Ch(\lFlat)$ are cofibrant and the trivially fibrant objects (the class $\mathcal R\cap \mathcal W$) are the injectives in $\Ch(\lFlat)$.

We now use Theorem \ref{locally_CB_rep} to obtain our statement. In fact the pure acyclic cochain complexes in $(\Ch(\cmod R),\mathcal E)$ correspond to the  complexes in the class $\widetilde{\lFlat}$ inside $(\Ch(\lFlat),\Ext|_{\Ch(\lFlat)})$. The injective model structure in $\Ch(\lFlat)$ is determined by the complete cotorsion pair $(\mathcal W, \mathcal R)$ above. The class $\mathcal W$ corresponds with the class of pure acyclic chain complexes in $(\Ch(\cmod R),\mathcal E)$. And the class $\mathcal R$ corresponds with the class of chain complexes $M$ such that any map $E\to M$ is null homotopic, for each pure acyclic chain complex $E$. Finally the class $\mathcal R\cap \mathcal W$ corresponds to the contractible complexes of pure-injective cartesian modules. Hence we get the corresponding injective model structure on $(\Ch(\cmod R),\mathcal E)$ as described.
\end{proof}

\begin{remark}
Given a strict representation $R$, there are corresponding versions of theorems \ref{theor.purei.env} and \ref{theor.der.pur} for the category $\mod R$.
\end{remark}

\appendix
\section{Purity in functor categories}

In this appendix we closely analyse the notion of purity in functor categories, reducing it to the notion of purity in categories of modules over unitary rings (with just one object): in fact, given a small preadditive category $\C$, we first reduce the theory of purity in $(\C^{op},\Ab)$ to the category of unitary modules over a ring with enough idempotents $R_\C$, and then to purity in a full subcategory of the category of modules over a ring $R_\C^*$ (with unit), that contains $R_\C$ as an ideal. Let us start recalling the following

\begin{definition}
Let $\C$ be a small preadditive category and let $N\leq M\in (\C^{op},\Ab)$. We say that $N$ is {\em pure} in $M$ if the sequence $0\to N\otimes_\C K\to M\otimes_\C K$ is exact for all $K\in (\C,\Ab)$.
\end{definition}
The reductions described above will allow us to deduce the following Theorem from results of \cite{flat_cover_conj}.

\begin{theorem}\label{card_purification}
Let $\C$ be a small preadditive category, let $\lambda=\max\{|\N|,|\mor \C|\}$.
Then, given $N\leq M\in(\C^{op},\Ab)$ with $|N|\leq \lambda$, there exists $N_*\leq M$ pure such that $N\leq N_*$ and $|N_*|\leq \lambda$. 
\end{theorem}

\subsection{From $(\C^{op},\Ab)$ to $\mod {R_\C}$}

Let $\C$ be a small preadditive category and consider the ring $R_\C$ which is the ring of $|\Ob  \C|\times |\Ob \C|$ matrices $(r_{ji})_{i,j\in \Ob\C}$ with $r_{ji}\colon i\to j$ and with a finite number of non-zero terms in each row and in each column. Given a morphism $\phi$ in $\C$, we identify $\phi$ with the matrix in $R_\C$ whose unique non-trivial entry is $\phi$. With this notation one can see that $\{\id_c:c\in \Ob\C\}$ is a complete set of pairwise orthogonal idempotents in $R_\C$.

\begin{definition}
A unitary right $R_\C$-module is an Abelian group $M$ together with a right $R_\C$ action
$$M\times R_\C\to M \ \ \ (m,r)\mapsto mr$$
such that, for all $m,n\in M$ and $r,s\in R_\C$:
\begin{enumerate}[\rm (1)]
\item $m(r+s)=mr+ms$;
\item $(m+n)r=mr+nr$;
\item $(mr)s=m(rs)$;
\item $MR_\C=\{mr:m\in M,r\in R_\C\}=M$.
\end{enumerate} 
Given two unitary right $R_\C$-modules $M$ and $N$, a homomorphism $\phi\colon M\to N$ is a homomorphism of right $R_\C$-modules if $\phi(mr)=\phi(m)r$ for all $m\in M$, $r\in R_\C$. We denote by $\mod {R_\C}$ the category of unitary right $R_\C$-modules. The category $\lmod {R_\C}$ of unitary left $R_\C$-modules is defined similarly. 
\end{definition}

From now on, unless explicitly stated, all the modules over $R_\C$ we consider are unitary. The following result is proved in \cite[Chapitre II, \S1]{Gabriel}:

\begin{proposition}
In the above notation, there are inverse equivalences of categories
$$S\colon (\C^{op},\Ab)\leftrightarrows \mod {R_\C}\colon T\,,$$
where $S$ is defined as follows:
\begin{enumerate}[\rm --]
\item given $F\colon \C^{op}\to \Ab$, $S(F)=\bigoplus_{c\in \Ob \C}F(c)$ with $(r_{dc})_{d,c}\in R_\C$ acting on $S(F)$ as the matrix $(F(r_{dc}))_{d,c}$, where $F(r_{dc})\colon F(d)\to F(c)$;
\item given $\alpha\colon F\rightarrow G$, $S(\alpha)\colon \bigoplus_{c\in \Ob \C}F(c)\to \bigoplus_{c\in \Ob \C}G(c)$ is the diagonal $\Ob\C$ matrix whose entry corresponding to $c\in \Ob \C$ is $\alpha_c$;
\end{enumerate}
while $T$ is defined as follows:
\begin{enumerate}[\rm --]
\item given a right $R_\C$-module $M$, $T(M)\colon \C^{op}\to \Ab$ is such that 
$$c\mapsto M\id_c \ \ \ (f\colon c\to d)\mapsto (M\id_d\overset{\cdot f}{\longrightarrow} M\id_c)$$
for all $c,d\in \Ob \C$ and $f\in \hom_\C(c,d)$;
\item given a homomorphism of right $R_\C$-modules, $\phi\colon M\to N$ we let $T(\phi)\colon T(M)\rightarrow T(N)$ be the natural transformation whose component at $c\in \Ob\C$ is  $T(\phi)_c\colon M\id_c\to N\id_c$, such that $T(\phi)_c(m\id_c)=\phi(m)\id_c$.
\end{enumerate}
\end{proposition}

Of course the above proposition has an analogous version for left modules. Abusing notations we use the same notations for the functors on left modules
$$S\colon (\C,\Ab)\leftrightarrows \lmod {R_\C}\colon T\,.$$

In the following lemma we show that in fact the above equivalences of categories respect the tensor product and so they can be used to translate problems about purity in the functor category $(\C^{op},\Ab)$ into analogous problems in $\mod R_\C$. Recall that the tensor product of (right and left) $R_\C$-modules is defined exactly as for modules over rings with $1$.

\begin{lemma}\label{pres_tens_1}
In the above notation, there is a canonical isomorphism of functors $(\C^{op},\Ab)\times (\C,\Ab)\to \Ab$:
$$-\otimes_{\C}-\cong S(-)\otimes_{R_\C} S(-)\,.$$
\end{lemma}
\begin{proof}
Given $M\in (\C^{op},\Ab)$ and $N\in (\C,\Ab)$, 
$$M\otimes_\C N=\left(\bigoplus_{c\in\Ob\C}M(c)\otimes_\Z N(c)\right)/H$$
where $H$ is the subgroup generated by the elements of the from $M(\alpha)(m)\otimes n-m\otimes N(\alpha)(n)$, with $m\in M(d)$, $n\in N(d)$ and $\alpha\in \hom_\C(c,d)$. To shorten notations we let $A=\bigoplus_{c\in\Ob\C}M(c)\otimes_\Z N(c)$, so that $M\otimes_\C N=A/H$. Similarly, 
$$S(M)\otimes_{R_\C} S(N)=\left(\bigoplus_{c\in \Ob\C}M(c)\otimes_\Z\bigoplus_{c\in \Ob\C}N(c)\right)/K=\left(\bigoplus_{c,d\in\Ob\C}M(c)\otimes_\Z N(d)\right)/K\,,$$
where $K$ is the subgroup generated by the elements of the from $xr\otimes y-x\otimes ry$, with $x\in S(M)$, $y\in S(N)$ and $r\in R_\C$. Again to shorten notations we let $B=\bigoplus_{c,d\in\Ob\C}M(c)\otimes_\Z N(d)$, so that $S(M)\otimes_{R_\C} S(N)=B/K$. Let $\phi\colon A\to B$ be the map that identifies $A$ with a sub-direct sum of $B$ in the obvious way and let $K'$ be the subgroup of $K$ generated by the elements $xr\otimes y-x\otimes ry$ with $x\in S(M)$, $y\in S(N)$ and $r=\id_c$ for some $c\in \Ob\C$. It is not difficult to see that $\phi$ induces an isomorphism $\pi\circ \phi=\bar\phi\colon A\to B/K'$, where $\pi$ is the projection $\pi\colon B\to B/K'$. In fact, $\bar\phi$ is surjective since, given $(m_c\otimes n_d)_{c,d}\in B$, 
$$(m_c\otimes n_d)_{c,d}-\phi(m_c\otimes n_c)_{c}=\sum_{c\in \Ob\C}((m_c\id_c\otimes n_d)_{c,d}-(m_c\otimes \id_c n_d)_{c,d}) \in K'$$ 
Similarly, if $\phi(m_c\otimes n_c)_c\in K'$, it means that there exist $x_1,\dots,x_n\in S(M)\,,\ y_1,\dots,y_n\in S(N)$ and $c_1,\dots, c_n\in \Ob\C$ such that  
$$\phi(m_c\otimes n_c)_c=\sum_{i=1}^n(x_i\id_{c_i}\otimes y_i-x_i\otimes \id_{c_i}y_i)\,.$$
Notice that the components corresponding to $(c,c)$, $c\in \Ob\C$, are zero on the right hand side, showing that $(m_c\otimes n_c)_c=0$. This proves the injectivity of $\phi$.\\
One can prove similarly that $\bar\phi(H)=K/K'$, showing that $\bar\phi$ induces an isomorphism
$$\phi_{M,N}\colon M\otimes_\C N\to S(M)\otimes_{R_\C} S(N)\,.$$
It is not difficult to show that $\phi_{M,N}$ is natural in both variables and so it gives an isomorphism of functors. 
\end{proof}

\subsection{From $\mod {R_\C}$ to $\sigma(R_\C)$}

Now we want to identify $\mod R_\C$ with a suitable full subcategory of $\mod R^*_\C$, where $R_\C^*$ is a ring (with unit) containing $R_\C$ as a subring. More precisely, as an Abelian group $R_\C^*=\Z\times R_\C$, while multiplication is defined as follows:
$$R_\C^*\times R_\C^*\to R_\C^*\ \ \ ((m,r),(n,s))\mapsto (mn,ms+rn+rs)\,.$$ 
It is easily seen that the above multiplication is associative and it is compatible with the sum in $\Z\times R_\C$, furthermore the element $(1,0)$ is a unit and the map $R_\C\to R_\C^*$ such that $r\mapsto (0,r)$ is an injective homomorphism of rings (without unit). The ring homomorphism $R_\C\to R_\C^*$ allows one to construct a forgetful functor
$$\mod R_\C^*\to \mod R_\C\,.$$
On the other hand, there is a canonical functor 
$$\mod R_\C\to \mod R_\C^*$$
that sends a right $R_\C$-module $M$ to the right $R_\C^*$-module whose underlying Abelian group is $M$ and 
$$M\times R_\C^*\to M\ \ \ (m,(n,r))=mn+mr\,.$$
In particular, $R_\C$ can be viewed naturally as a right $R_\C^*$-module.

\begin{definition}
In the above notation, we let $\sigma(R_\C)=\{M\in\mod R_\C^*: \exists R_C^{(I)}\to M\to 0\}$ be the full subcategory of $\mod R_\C^*$ of all the modules generated by $R_\C$. Analogously one defines $(R_\C)\sigma\subseteq \lmod{R_\C^*}$.
\end{definition}

Notice also that, given a morphism $\phi\colon M\to N$ of right $R_\C$-modules, $\phi$ is an endomorphism of $R_\C^*$-modules when we endow $M$ and $N$ with their canonical right $R_\C^*$-module structure as described above. 

The following result is proved in \cite[\S49]{wisbauer}:

\begin{proposition}
In the above notation, there is an equivalence of categories $U\colon \mod R_\C\to \sigma(R_\C)$, assigning to $M\in \mod R_\C$ the corresponding $R_\C^*$-module $M_{R_\C^*}$. Furthermore, given a right $R_\C$-module $M$,
\begin{enumerate}[\rm (1)]
\item $U$ induces a bijection between the lattice of $R_\C$-submodules of $M$ and that of $R_\C^*$-submodules of $U(M)$;
\item $|M|=|U(M)|$ (as $M$ and $U(M)$ have the same structure as Abelian groups).
\end{enumerate}
\end{proposition}

Of course the above proposition has its analog for left module. We denote again by 
$$U\colon \lmod {R_\C}\to (R_\C)\sigma$$
the equivalence of categories on the left. 

\begin{lemma}\label{pres_tens_2}
In the above notation, there is a canonical isomorphism of functors $\mod {R_\C}\times \lmod{R_\C}\to \Ab$:
$$-\otimes_{R_\C}-\cong U(-)\otimes_{R^*_\C} U(-)\,.$$
\end{lemma}
\begin{proof}
Let $M\in \mod {R_\C}$ and $N\in \lmod{R_\C}$. Then,
$$M\otimes_{R_\C}N=(M\otimes_\Z N)/H \text{ and }U(M)\otimes_{R^*_\C}U(N)=(M\otimes_\Z N)/K\,,$$
where $H=\langle mr\otimes n-m\otimes rn\colon r\in R_\C,m\in M,n\in N \rangle$ and $K=\langle mr\otimes n-m\otimes rn\colon r\in R^*_\C,m\in M,n\in N \rangle$. Thus, it is enough to show that $H=K$ as subgroups of $M\otimes_\Z N$. Since $R_\C$ is a subring of $R^*_\C$, it is clear that $H\subseteq K$. On the other hand, given $(k,r)\in R^*_\C$, $m\in M$ and $n\in N$ we have to show that $m(k,r)\otimes n-m\otimes (k,r)n\in H$. But in fact, $(k,r)=(k,0)+(0,r)$ and so
\begin{align*}m(k,r)\otimes n-m\otimes (k,r)n&=m(k,0)\otimes n+m(0,r)\otimes n-m\otimes (k,0)n-m\otimes (0,r)n\\
&=(m(k,0)\otimes n-m\otimes (k,0)n)+(m(0,r)\otimes n-m\otimes (0,r)n)\\
&=0+(m(0,r)\otimes n-m\otimes (0,r)n)\in H\,.
\end{align*}
On can verify that the isomorphism $M\otimes_{R_\C}N\to U(M)\otimes_{R^*_\C}U(N)$ such that $m\otimes n\mapsto m\otimes n$ is natural in both variables.
\end{proof}

Now we have all the tools to give the following

\begin{proof}[Proof of Theorem \ref{card_purification}]
Given $N\leq M$ as in the statement, consider $US(N)\leq US(M)$ in $\sigma(R_\C)$. Notice that $|R_\C^*|=\max\{|\N|,|R_\C|\}=\max\{|\N|,\mor \C\}$, so $|R_\C^*|\leq \lambda$, and $|US(N)|=|S(N)|=|N|\leq \lambda$.Thus, we can use \cite[Lemma 1]{flat_cover_conj} to show that there is a pure submodule $K_*$ of $US(M)$ such that $|K_*|\leq \lambda$ and $US(N)$. Since both $U$ and $S$ induce isomorphisms of lattices of submodules and preserve cardinality, we obtain $N_*\leq M$ such that $N\leq N_*$ and $|N_*|\leq \lambda$. To conclude it is enough to verify that $N_*$ is pure in $M$, but this is easy to see using Lemmas \ref{pres_tens_1} and \ref{pres_tens_2}. 
\end{proof}

\medskip\noindent
{\bf Acknowledgement.} {We would like to thank Professor Haynes Miller for helpful comments and suggestions.}

%Address of the author:

%Simone Virili - {\tt simone@mat.uab.cat}
%
%Departament de Matem\`atiques, Universitat Aut\`onoma de Barcelona
%
%Edifici C - 08193 Bellaterra (Barcelona), Spain.

\end{document}